\documentclass[11pt,oneside,letter]{article}
\usepackage[top=1 in, bottom=1 in, left=0.85 in, right=0.85 in]{geometry}

\usepackage{tikz}
\usetikzlibrary{shapes.geometric, arrows.meta, positioning}

\tikzstyle{block} = [rectangle, draw, rounded corners, text centered, minimum height=2em, minimum width=4cm]
\tikzstyle{arrow} = [thick, ->, >=Stealth]

\usepackage{enumitem} 
\usepackage{amsmath} 
\usepackage{amsthm} 
\usepackage{amssymb}
\usepackage{tabularx}
\usepackage{graphicx} 
\usepackage{subcaption}
\usepackage{algorithm}
\usepackage{graphicx}
\usepackage{epstopdf}
\usepackage{enumerate}
\usepackage{stfloats}
\usepackage{xcolor}
\usepackage{thmtools}
\usepackage{thm-restate}
\usepackage{multirow}
\usepackage[hidelinks]{hyperref}

\newtheorem{theorem}{Theorem}
\newtheorem{lemma}{Lemma}

\newtheorem{assumption}{Assumption}
\newtheorem{remark}{Remark}

\newcommand{\delete}[1]{}

\title{Zero-Waiting Load Balancing with Heterogeneous Servers in \\
Heavy Traffic} 
\author{Xin Liu \\ShanghaiTech University \\ liuxin7@shanghaitech.edu.cn
\and Lei Ying \\University of Michigan, Ann Arbor \\ leiying@umich.edu}
\date{}

\begin{document}
\maketitle

\begin{abstract}
We study the steady-state delay performance of load balancing in large-scale systems with heterogeneous servers in the heavy-traffic regimes. The system consists of $N$ servers, each with a local buffer of size $b-1$, serving jobs in the first-in-first-out (FIFO) order. Jobs arrive according to a Poisson process with rate $\lambda N$, where $\lambda = 1 - N^{-\alpha}$ for any $\alpha \in (0,1)$. Service times are assumed to be exponentially distributed with fully heterogeneous rates, where the service rate of each server can differ and may scale with the system size $N$.
We study a queue length aware and service rate aware load balancing policy, Join-the-Fastest-Shortest-Queue (JFSQ), and demonstrate that it achieves {\it asymptotic zero waiting time and probability} under the heavy traffic regimes, including both the Sub-Halfin-Whitt ($\alpha \in (0,0.5)$) and Super-Halfin-Whitt ($\alpha \in [0.5,1)$) regimes. The performance bounds of waiting time and probability explicitly capture the convergence rate w.r.t. the system size $N$ and show the negative effect of server heterogeneity. 
Our analysis builds on the general framework of Stein’s method with iterative state-space peeling, where we design a sequence of 
Lyapunov functions to analyze the high-dimensional heterogeneous system without assuming exchangeability and monotonicity.  
Our analysis shows that JFSQ efficiently utilizes servers with higher capacities, and the steady-state system can be coupled with a single-server queue via Stein's method. 
To the best of our knowledge, this is the first work to establish  delay performance bounds of a load-balancing system with size $N$ and fully heterogeneous servers in heavy traffic. 
\end{abstract}

\section{Introduction}

Modern large-scale computing platforms, including data centers and cloud platforms, rely on efficient load balancing to distribute incoming jobs to appropriate servers with the objective of optimizing delay performance. Load balancing has become a cornerstone of large-scale server systems, and its design and analysis have been extensively studied in both queueing theory and performance modeling \cite{Mit_96,VveDobKar_96,LuXieKli_11, harchol-balter_2013}. Join-the-Shortest-Queue (JSQ) policy has long been regarded as a gold standard load balancing due to its excellent delay performance, where each incoming job is routed to the server with the shortest queue. The study of JSQ traces back to foundational work in the 1970s \cite{Win_77,Web_78}, and gained renewed interest in the past two decades with the advent of data centers, cloud platforms with flooding AI/ML workloads \cite{KwoLiZhu_23, FuXueHua_24, KunAnjAnk_25}. The early studies focused on fixed and small-server settings, whereas more recent efforts have studied JSQ in large-scale, many-server systems. These studies concentrated on showing JSQ policy can achieve near-optimal delay performance in large-scale {\it homogeneous} systems even in the heavy traffic regimes \cite{EscGam_18,Bra_20}. Despite the inspiring theoretical results and strong performance guarantees of JSQ in homogeneous systems, these results do not extend to settings where servers are {\it heterogeneous}. 
In a  heterogeneous setting, the classical JSQ policy, which ignores service heterogeneity, can perform poorly, not to mention achieving low delay performance \cite{GarJalWic_21}. 
To address the heterogeneity, most existing works partition servers into a finite number of service classes/groups, assuming that service rates are identical within a class, finite, and independent of system size \cite{VanCom_21,GarJalWic_21,AbuDorGar_22,AbuDorGar_24,BhaBukMuk_25}. This structural assumption is crucial in studying load-balancing systems using mean-field models or diffusion models and establishing asymptotic zero-delay performance. 
However, these assumptions do not capture the full heterogeneity in real-world server systems where the servers often vary substantially in service rates due to differences in hardware generations and software configurations \cite{SubArfLin_23, UmOhKan_24}.  
Besides, the existing studies often focus on the critical load regimes such that the nominal idle capacity is sufficient to cancel the negative effects of server heterogeneity \cite{Sto_15}.  
Therefore, this motivates a fundamental question:

\vspace{3pt}
{\it Can we establish zero-waiting load balancing for fully heterogeneous systems in heavy-traffic regimes, where the service rates vary across servers and the service rate of a server potentially scales with $N$?}
\vspace{3pt}

To address this question, we study a natural heterogeneous generalization of JSQ, namely the Join-the-Fastest-Shortest-Queue (JFSQ) policy as in \cite{WenZhoSri_20, BhaBukMuk_25}, in a setting with fully heterogeneous servers and in heavy-traffic regimes, including both sub and super Halfin-Whitt regimes. Our main results show that JFSQ achieves asymptotically zero waiting at steady state, with explicit system size bounds that characterize the convergence rates. The performance bounds reveal how system behavior depends on the heterogeneity level.  
To establish these results, we build on Stein’s method with iterative state-space peeling (ISSP) \cite{LiuYin_20, LiuKanYin_22, LiuKanYin_24}, a powerful tool for steady-state analysis of load balancing systems. However, applying it to our heterogeneous setting requires several technical innovations due to the high-dimensional and non-exchangeable state representation induced by the challenges of full heterogeneity. We directly work on the original high-dimensional state space and design a sequence of novel Lyapunov functions to show that, under JFSQ, that the servers with higher capacities are sufficiently utilized, and the steady-state system concentrates around the state space in which the total capacity of busy servers exceeds the total arrival rate. 
With these concentration results, we introduce a unified coupling argument that relates JFSQ to the fluid model of a single-server queue across both sub and super Halfin-Whitt regimes, and explicitly quantifies the delay performance. We believe these novel methodologies can be of independent interest and applied to other heterogeneous queueing systems.

\subsection*{Related Work}

The performance analysis of load balancing in many-server systems has a long and rich history. For homogeneous systems, where all servers are identical, a sequence of classical load balancing policies has been proposed and analyzed, including join-the-shortest-queue (JSQ), join-the-idle-queue (JIQ) \cite{LuXieKli_11,Sto_15}, and power-of-$d$ (Po$d$) (also called JSQ-$d$) \cite{Mit_96,VveDobKar_96}.
The early foundational works \cite{Win_77,Web_78} proved that JSQ is delay-optimal in finite-server systems assuming Poisson arrivals and exponential service times under critical load (i.e., the arrival rate is strictly smaller than the service rate and its gap is constant independent of the system size). For the many-server asymptotic regime, a recent important work \cite{EscGam_18} analyzed JSQ in the Halfin-Whitt regime and shows that the diffusion-scaled process converges to a two-dimensional diffusion limit. The corresponding steady-state performance was analyzed in \cite{Bra_20} via Stein's method, and the further refined characteristics (e.g., tail probability and sensitivity) have been studied in \cite{BanMuk_19,BanMuk_20}. 
The JIQ policy has also attracted attention due to its low communication overhead \cite{LuXieKli_11, Sto_15}. In \cite{Sto_15}, it was shown that JIQ achieves asymptotically vanishing delay 
in large-scale heterogeneous systems, though their service rates are independent of system size $N$ under the critical load regimes.  
For Po$d$, the diffusion limit of Po$d$ is shown to converge to that of JSQ in the Halfin-Whitt regime at the process level when $d$ is in the order of $\sqrt{N}\log N$ in work \cite{MukBorvan_18}. \cite{LiuYin_20,LiuYin_21} established that a class of load balancing policies, including JSQ, JIQ, and Po$d$ with properly chosen $d$, achieve zero-waiting performance at steady state in both the sub and super Halfin-Whitt regimes. The exact asymptotic queue distribution in the super-Halfin-Whitt regime was later established in \cite{ZhaBanMuk_25}. These results, however, largely rely on the assumption that all servers are homogeneous. In contrast, modern data centers and cloud platforms often exhibit substantial heterogeneity in server speeds due to differences in hardware generations and software configuration. In these heterogeneous systems, heterogeneity-unaware (or server rate-unaware) load balancing can lead to substantial inefficiencies, as jobs may be assigned to slow servers with short queues while faster servers remain underutilized. 
Even the classical JSQ may perform poorly \cite{GarJalWic_21}. 
This inefficiency has motivated the design of \emph{heterogeneity-aware or service-aware} load balancing policies that incorporate service rates into the decision process.

Almost all existing works address server heterogeneity by partitioning servers into a finite number of classes or groups \cite{VanCom_21, GarJalWic_21, AbuDorGar_22,BhaMuk_22,AbuDorGar_24,BhaBukMuk_25} or assume the service rates do not scale with $N$ \cite{LuoZub_25}. For example, \cite{VanCom_21} studied a system with two groups of servers and proposed a randomized load balancing policy based on buffer availability, which admits a product-form stationary distribution. Similarly, \cite{GarJalWic_21} introduced heterogeneity-aware variants of JSQ and JIQ for grouped heterogeneous systems, where jobs are probabilistically assigned to a group and then routed within that group using JSQ or JIQ. This design is further generalized into the Class and Idleness/Length Differentiated Assignment policies \cite{AbuDorGar_22, AbuDorGar_24}. Under mean-field independence assumptions, these works formulate and solve an optimization problem to determine optimal group-assignment probabilities that minimize mean response time, and demonstrate the  delay improvement through numerical experiments. There also exists related work on distributed and stochastic coordination in heterogeneous systems under critical load regimes \cite{GorVarMos_21, DaaMarDeb_25}. In \cite{GorVarMos_21}, the authors studied distributed load balancing with multiple local dispatchers and formulated the problem as a stochastic optimization in a heterogeneous setting. Their proposed algorithm leverages local, distributed coordination among dispatchers to approximate the optimal assignment, though no queueing delay analysis was reported. In \cite{DaaMarDeb_25}, the authors investigated distributed rate scaling in the heterogeneous systems, where each server dynamically adjusts its service rate to minimize a global cost of the system response time. They proposed a distributed gradient-descent type policy and showed that it minimizes the global cost under a mean-field approximation. 
The most closely related works are \cite{BhaMuk_22} and \cite{BhaBukMuk_25}, which analyze the Join-the-Fastest-Shortest-Queue (JFSQ) policy in heterogeneous-server systems under the critical load regime and the Halfin-Whitt regime, respectively. In \cite{BhaMuk_22}, the authors prove that JFSQ is asymptotically delay-optimal in the critical regime. This result is extended in \cite{BhaBukMuk_25}, where it is shown that JFSQ retains asymptotically delay optimality in the Halfin-Whitt setting by leveraging diffusion-limit analysis and Stein's method. Specifically, they establish that the diffusion-scaled queue length process converges to a two-dimensional reflected Ornstein–Uhlenbeck (OU) process, and use Stein’s method to prove convergence of the stationary distribution. It is important to note, however, that both papers assume the number of server types is finite and independent of $N.$ 

Our work departs from the above in several important ways. First, we consider fully heterogeneous systems without assuming any grouping of servers, and allow the service rates to be system-size dependent. Second, we analyze the JFSQ policy in a broad heavy-traffic regime that encompasses both sub and super Halfin-Whitt regimes. Third, we prove that JFSQ achieves asymptotic zero waiting results, where the performance bounds provide an explicit convergence rate and also quantify the impact of heterogeneity on delay performance via an explicit dependence on the service rates. To the best of our knowledge, this is the first work to provide such steady-state delay bounds at steady state for load balancing in fully heterogeneous many-server systems in the heavy traffic regimes.

\section{Model and Main Results}
We consider a many-server system with $N$ heterogeneous servers, where jobs arrive according to a Poisson process with rate $\lambda N$, and the service time at server $n$ is exponentially distributed with rate $1/\mu_n$. Without loss of generality\footnote{This assumption is made for normalization. If instead $\sum_{n=1}^N \mu_n = cN$ for some constant $c > 0$, we can equivalently rescale the arrival and service rates to $\lambda / c$ and $\mu_n / c,$ respectively, so that the total capacity becomes $N.$}, the service rates satisfy $\sum_{n=1}^N \mu_n = N$. {We study the many-server system in heavy-traffic regimes, where the system load approaches capacity as $\lambda_N = 1 - \beta N^{-\alpha}$ for any constant $\beta > 0$. We let $\beta = 1$ for notational simplicity and also write $\lambda$ instead of $\lambda_N$ when there is no ambiguity.}
The well-known case of $\alpha = \frac{1}{2}$ corresponds to the classical \emph{Halfin-Whitt regime}~\cite{HalWhi_81}, also called quality-and-efficiency-driven (QED) regime. The regime $\alpha < \frac{1}{2}$ and $\alpha \geq \frac{1}{2}$ are named as \emph{Sub-Halfin-Whitt regime} and \emph{Super-Halfin-Whitt regime}, respectively. As shown in Figure \ref{model-pod}, each server maintains a separate queue with the buffer size $b-1$ (i.e., each server can have one job in service and $b-1$ jobs in queue).
\begin{figure}[!htbp]
  \centering
  \includegraphics[width=3.7in]{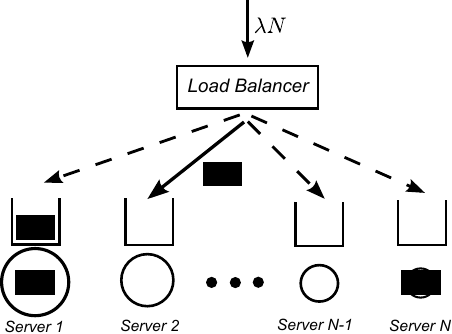}
  \caption{JFSQ Load Balancing in Many-Server Systems.}
  \label{model-pod}
\end{figure}
{Define 
$\mathbb{Q} := \left\{ (q_1,\dots,q_N) \in \mathbb{Z}^N \;:\; 0 \le q_i \le b,\ \forall i=1,\dots,N \right\}.$}
The system state can be described by the vector $\mathbf{Q}(t)=(Q_1(t),\dots,Q_N(t)) \in \mathbb Q,$ where $Q_i(t)$ is the queue length of server $i$ at time $t$, including the job in service. We study the \emph{Join-the-Fastest-Shortest-Queue} (JFSQ) load balancing policy, a state-aware and service rate-aware policy, where the load balancer has full knowledge of the current queue lengths $\mathbf{Q}(t)$ and the service rates $\{\mu_n\}$. In JFSQ, a job is routed to a server with the shortest queue length. If there are multiple such servers, the job is assigned to the one with the highest service rate. If the service rates are tied, the dispatcher selects one uniformly at random. We assume that servers are indexed in descending order of service rate, i.e., $\mu_1 \geq \mu_2 \geq \cdots \geq \mu_N \geq \epsilon$, where $\epsilon > 0$ is a constant independent of $N$. Note that if $\epsilon =o(1),$ then the service time of a job on server $n$ goes to infinity as $N$ increases\footnote{ We follow the standard big-O notation.  
For functions $f(N)$ and $g(N)>0$, we write $f(N)=O(g(N))$ if there exist constants $C>0$ and $N_0$ such that $|f(N)|\le C g(N)$ for all $N\ge N_0$; 
$f(N)=o(g(N))$ if $\lim_{N\to\infty} f(N)/g(N)=0$; and 
$f(N)=\Theta(g(N))$ if there exist constants $C_1,C_2>0$ and $N_0$ such that $C_1 g(N)\le f(N)\le C_2 g(N)$ for all $N\ge N_0$.}. We assume no such slow servers in the system, but it will become clear that this assumption can be relaxed. 

Under JFSQ policy, the queue-length process $\{\mathbf{Q}(t)\}$ evolves as a continuous-time Markov chain (CTMC). We are interested in its delay performance in steady state. Let $Q_i$ denote the steady-state queue length at server $i$, i.e., $Q_i \sim Q_i(\infty)$. {Our objective is to determine whether the JFSQ policy achieves “zero-waiting” performance as the system size increases, and to understand how server heterogeneity affects the resulting delay performance.}
For tractability, we ignore integer rounding issues and assume the existence of a critical index $N_2$ such that the aggregate total capacity of the fastest $N_2$ servers exactly supports the total arrival rate 
\begin{align*}
\lambda N = \sum_{n=1}^{N_2} \mu_n.  
\end{align*}

To present our main results, we impose several assumptions on the distribution of service rates. These conditions quantify the degree of heterogeneity and its dependence on the system size $N$. In particular, the assumptions focus on constraints for $\mu_1$, where a more restrictive condition on $\mu_1$ (i.e., allowing smaller $\mu_1$) is required in the Super–Halfin–Whitt regime than that in the Sub–Halfin–Whitt regime. This is intuitive because the system can tolerate less heterogeneity under a heavier traffic load.   
\begin{assumption}[Conditions on \(N\) in Sub-Haffin-Whitt regime]
Assume $\alpha \in (0, \tfrac{1}{2})$ and $b \geq 5$. {For any $r \in Z^{+},$} we assume that the system size $N$ satisfies
\begin{align*}
&\max\left\{
\frac{2(r+1)}{\log N},\
\frac{27(r+1)b\log N}{N^{0.25-0.5\alpha}}
\right\}
\ \leq\ 
\epsilon
\ \leq\
0.5, ~ \frac{\mu_1}{\epsilon^2} \leq \frac{N^{1-\alpha}}{64b(2\log^2 N + 3)}.
\end{align*} \label{assumption: sub}
\end{assumption}

\begin{assumption}[Conditions on \(N\) in Super-Haffin-Whitt regime]
Assume $\alpha \in [\tfrac{1}{2}, 1)$ and $b \geq 5$. {For any $r \in Z^{+},$} we assume that the system size $N$ satisfies
\begin{align*}
\frac{\mu_1}{\epsilon} \leq \frac{N^{1-\alpha}}{24rb\log^2 N}. 
\end{align*} \label{assumption: super}
\end{assumption}
Next, we formalize the discussion above and proved the asymptotic zero waiting results. 
Let $\mathcal W$ denote the event that an incoming job is routed to a busy server, and $p_{\mathcal W}$ denote the probability of this event at steady-state. Let $\mathcal B$ denote the event that an incoming job is blocked (discarded) and $p_{\mathcal B}$ denote the probability of this event at steady-state. {Note the  ${\mathcal B}\subseteq {\mathcal W}$ because an incoming job is blocked when being routed to a busy server with $b$ jobs.} 
Furthermore, let $W$ denote the  waiting time of the jobs that are not blocked at steady-state. 
We denote two important parameters that capture server heterogeneity:
\begin{align*}
k_{\mathrm{sub}} = \frac{16(b-1)\mu_{N_2}}{\epsilon^2}~~\textup{and}~~ 
k_{\mathrm{super}} = \frac{24r(b-1)\mu_1}{\epsilon},
\end{align*}
which correspond to the sub-Halfin--Whitt and super-Halfin--Whitt regimes, respectively. 
\begin{theorem}\label{thm:zerodelay}
Assume the service rates $\{\mu_n\}$ satisfy Assumptions~\ref{assumption: sub} and~\ref{assumption: super}. Then under the JFSQ policy, the waiting time and probability at steady-state have the following performance bounds: 
\begin{itemize}[leftmargin=*]
\item[1)]  {Sub-Halfin-Whitt regime:} 
we have
\begin{align*}
\mathbb{E}[W] = \mathcal{O}\left( \frac{1}{N^\alpha} + \frac{k_{sub}}{N^{1 - \alpha}} \right), &~~
p_{\mathcal{W}} = \mathcal{O}\left( \frac{1}{N^\alpha} + \frac{k_{sub}}{N^{1-\alpha}}\right).
\end{align*}

\item[2)] {Super-Halfin-Whitt regime:} 
we have
\begin{align*}
\mathbb{E}[W] = \mathcal{O}\left( \frac{k_{super} \log N}{N^{1 - \alpha}} + \frac{\log N}{\sqrt{N}} \right), &~~
p_{\mathcal{W}} = \mathcal{O}\left( \frac{k_{super} \log N}{N^{1 - \alpha}} \right).
\end{align*}
\end{itemize}
\end{theorem}
Theorem~\ref{thm:zerodelay} establishes the asymptotic zero-waiting under the JFSQ policy, i.e., both the expected waiting time and the waiting probability vanish as the system size $N$ grows, in both sub and super Halfin-Whitt regimes. The presence of $k_{sub}$ and $k_{super}$ in the bounds explicitly captures the effect of heterogeneity: larger disparities between fast and slow servers (i.e., large $\mu_{N_2}/\epsilon$ or $\mu_1/\epsilon$) lead to larger factors and hence larger waiting time and probability. Let $\bar Q = \frac{1}{N}\sum_{n=1}^{N} Q_n$ be the average queue length at steady-state.  To prove Theorem~\ref{thm:zerodelay}, we introduce the following key results on the higher moments of the average queue length under JFSQ. 
\begin{theorem} \label{thm:main}
Assume $\lambda=1- N^{-\alpha}$ for $0<\alpha<1$ and buffer size $b-1.$ 
Given any positive integer $r$ and when the service rates $\{\mu_{n}\}$ satisfy Assumptions \ref{assumption: sub} and \ref{assumption: super}, the following bound holds under JFSQ

\noindent 1) Sub-Haffin-Whitt regime: 
we have  
\begin{align}
\mathbb E\left[\left(\max\left\{\bar Q - \frac{N_2}{N} - \left(1-\frac{4}{k_{sub}\epsilon}\right)\frac{1}{N^{\alpha}}, 0\right\}\right)^r\right]
\leq {17\left( \frac{ 2k_{sub} r}{N^{1-\alpha}} \right)^r.} 
\end{align}  

\noindent 2) Super-Haffin-Whitt regime:  
we have  
\begin{align}
\mathbb E\left[\left(\max\left\{\bar Q -1-\frac{k_{super}\log N}{N^{1-\alpha}}, 0\right\}\right)^r\right]\leq 10 \left(\frac{2r}{N^{1-\alpha}}\right)^r.
\end{align} 
\end{theorem}
Based on Theorem~\ref{thm:main}, {by setting $r=1,$} we establish the upper bound of the average queue length under heavy-traffic regimes.
\begin{itemize}[leftmargin=*]
    \item [1)] Sub-Halfin-Whitt regime:  
    we have
    \[
    \mathbb{E}[\bar{Q}] \leq \frac{N_2}{N} + \left(1 - \frac{4}{k_{sub} \epsilon}\right)\frac{1}{N^\alpha} + \frac{34k_{sub}}{N^{1 - \alpha}},
    \]
    {which implies that the average queue length per server is upper bounded by a quantity that scales as $\frac{N_2}{N}$ as $N \to \infty$. In the homogeneous case, where $N_2 = \lambda N$, this yields \(\mathbb{E}[\bar{Q}] \le \lambda + o(1),\) indicating that the average time a job spends in the system is asymptotically controlled by the normalized system load.}
    In the heterogeneous setting, delay performance is influenced by the parameter \( k_{sub} \), which depends on the degree of heterogeneity. Specifically, if \( \mu_{N_2}/\epsilon \) is large (i.e., large \( k_{sub} \) and \( k_{sub} \epsilon \)), the average queue length becomes large.
    
    \item [2)] Super-Halfin-Whitt regime: 
    we have
    \[
    \mathbb{E}[\bar{Q}] \leq 1 + \frac{k_{super} \log N + 20}{N^{1 - \alpha}},  
    \]
    {which implies that the average queue length per server is upper bounded 1 as \( N \to \infty \).}
    Thus, all servers remain busy in the limit, and the average number of jobs waiting in the queue vanishes. As in the Sub-Halfin-Whitt regime, heterogeneity impacts queue lengths through \( k_{super} \). When \( \mu_1 / \epsilon \) becomes large, \( k_{super} \) also increases, leading to a larger average queue length.
\end{itemize}

\section{Theoretical Analysis}
Recall we assume that the servers are ordered in a descending order according to their service rates, i.e., $\mu_1\geq \mu_2\geq \cdots \geq \mu_N.$ In our analysis, we implicitly divide the servers into multiple groups according to their aggregated service capacity. 
Without loss of generality, let $\delta > 0$ and we define $N_1(\delta)$ such that $$\sum_{n=1}^{N_1(\delta)} \mu_n = (1-\delta)N.$$ We occasionally omit $N_1(\delta)$ as $N_1$ without any confusion. 
Recall given a positive $\delta>0,$ $\lambda N := N - N^{1-\alpha} \geq  (1-\delta)N$ holds for a sufficient large $N$ such that $\delta \geq N^{-\alpha}$. Intuitively, under load balancing algorithms that prefer fast servers, e.g., JFSQ, these servers are always busy with a high probability. A similar intuition holds for the first $N_2$ fastest servers because  $\sum_{n=1}^{N_2} \mu_n = \lambda N.$ We will start with this intuition and iteratively prove the asymptotic zero waiting results in Theorems \ref{thm:zerodelay} and \ref{thm:main}. We first provide a roadmap to illustrate our analysis. 
\begin{figure}[!htbp]
  \centering
  \includegraphics[width=2.9in]{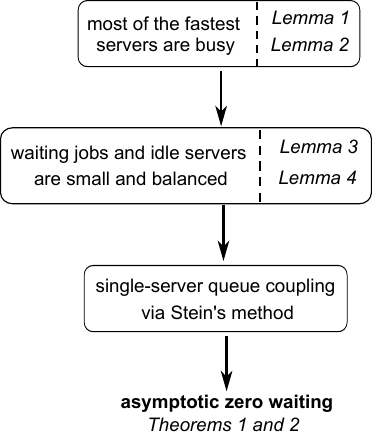}
  \caption{Roadmap for achieving asymptotic zero waiting results.}
  \label{roadmap}
\end{figure}

{\bf Roadmap to achieve asymptotic zero-waiting:} 
Our goal is to establish that the JFSQ policy achieves asymptotically vanishing waiting in heterogeneous many-server systems in the heavy traffic regimes. The analysis proceeds through a sequence of concentration results via Lyapunov drift analysis \cite{ErySri_12, WanMagSri_21}. We begin by showing that the fastest servers are almost always busy under JFSQ in Lemmas \ref{lem:most n1} and \ref{lem:most n2}. This ensures that the most fast servers are fully utilized. Next, we establish a form of state-space collapse in Lemmas \ref{lem:SSC-sup} and \ref{lem:SSC-super} and for sub and super Halfin-Whitt regimes: the system's state concentrates in regions where either the number of idle servers or the number of waiting jobs is small, but not both. This balancing behavior is captured via a unified Lyapunov function. These structural results allow us to approximate the system behavior by that of a single-server queue with an effective service rate slightly above the arrival rate. Using Stein’s method, we derive a stochastic coupling between a high-dimensional load balancing system and a one-dimensional single-server system. This coupling leads to high-order moment bounds on average queue length and asymptotic zero waiting results, as stated in Theorems \ref{thm:zerodelay} and \ref{thm:main}.

\subsection{Most of the fastest servers are busy}
To justify the intuition that ``most of the fastest servers are busy'' under JFSQ, we design Lyapunov functions that are directly related to the number of idle servers in two groups defined by $N_1$ and $N_2$. Intuitively, under JFSQ, the idle servers in the fastest $N_1$ and $N_2$ servers are extremely small as shown in Figure \ref{fig:lem12}. The following two lemmas will further show that most idle servers among the first $N_2$ servers are those indexed between $N_1+1$ and $N_2,$ i.e., not the fastest $N_1$ servers. 

\begin{figure}[!htbp]
  \centering
\includegraphics[width=4.7in]{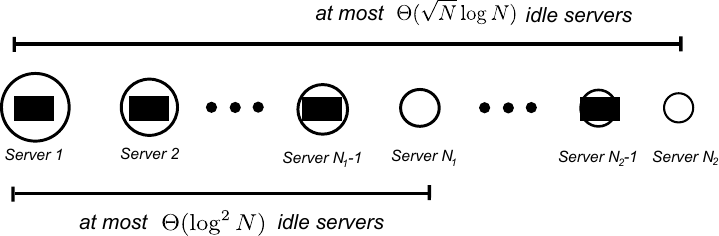}
  \caption{Most capable servers are busy  under JFSQ.}
  \label{fig:lem12}
\end{figure}

\begin{restatable}{lemma}{mostnone}\label{lem:most n1}
Define a Lyapunov function to be $$V(q)=\sum_{n=1}^{N_1(\delta)}\mathbb{I}(q_n=0).$$ For any $\delta$ satisfying $\frac{2}{N^\alpha} \leq \delta\leq 0.5,$ we have 
\begin{align*}
  \mathbb P\left(V(Q)\leq 1 + 8\log^2 N \right)&\geq 1- \exp\left(-{\delta}\log^2 N\right).
\end{align*}
\end{restatable}
The lemma shows that, with high probability, the number of idle servers is at most $\Theta(\log^2 N)$  among the $N_1$ fastest servers under JFSQ policy. Further, we bound the idle servers for the $N_2$ fastest servers.
\begin{restatable}{lemma}{mostntwo}\label{lem:most n2}
Define a Lyapunov function to be $$V(q)=\sum_{n=1}^{N_2}\mathbb{I}(q_n=0).$$
For $\mu_{N_2} \leq \frac{\sqrt{N}}{\log N},$ we have 
\begin{align*}
  \mathbb P\left(V(Q)\leq 5\sqrt{N} \log N \right)\geq 1-\exp\left(-\epsilon\log^2 N\right).
\end{align*}
\end{restatable}
The lemma shows that, with high probability, the number of idle servers is at most $\Theta(\sqrt{N}\log N)$ servers among the $N_2$ fastest servers under JFSQ policy. One might be curious why we establish two ``similar'' results. This is mainly because we want to show that most idle servers among the first $N_2$ servers are those indexed from $N_1+1$ to $N_2,$ i.e. not among the first $N_1$ servers, which helps us to lower bound the total service rate from busy servers.    

To proceed further, we define the key high probability event as follows 
\begin{align}
\mathcal E_{0}=&\left\{Q~|~ 
 \sum_{n=1}^{N_1}\mathbb{I}(Q_n = 0) \leq 1 + 8\log^2 N ~~\text{and}~~ \sum_{n=1}^{N_2}\mathbb{I}(Q_n = 0) \leq 5\sqrt{N}\log N \right\} \label{eq: most}
\end{align}
This implies that under the JFSQ policy, most of the fastest servers are busy at steady state, confirming that JFSQ efficiently utilizes the highest-capacity part of the system. 
Specifically, given $\mathcal E_0,$ the total number of busy servers $\sum_{n=1}^{N} \mathbb I(Q_n \geq 1) = N - \sum_{n=1}^{N} \mathbb I(Q_n = 0)$ is at least $N_2-5\sqrt{N}\log N$. However, this may be insufficient to conclude that JFSQ can achieve asymptotic zero waiting. To illustrate this, we consider the homogeneous setting with $N_2 = \lambda N$, where the total service rate is above $\lambda N-5\sqrt{N}\log N$, even smaller than the arrival rate $\lambda N$. Therefore, we need to extract a more refined behavior of total busy/idle servers to show asymptotic zero waiting under JFSQ, as we introduce next. 

\subsection{State-space Collapse in the Heavy-traffic Regime}
In this section, we establish one of the central results of the paper—\emph{state-space collapse}—which forms the foundation for proving our main performance bounds. Intuitively, we show that the system's states concentrate in a region where either the number of idle servers or the number of waiting jobs is small, as shown in Figure \ref{fig:SSC}. In other words, the system cannot have many idle servers and waiting jobs, simultaneously. To formalize this idea, we introduce a unified Lyapunov function that applies to both Sub-Halfin-Whitt regime ($\alpha \in (0, 0.5)$) and Super-Halfin-Whitt regime ($\alpha \in [0.5, 1)$). The structure of the Lyapunov function is the same across both regimes, differing only in the values of $A$ and $B$:
\begin{align}
V(q) = \min\left\{ \sum_{n=1}^{N} (q_n - 1)\mathbb{I}(q_n \geq 2) - A,\ B - \sum_{n=1}^{N} \mathbb{I}(q_n \geq 1) \right\}. \label{eq:ly-func-unified}
\end{align}
{We remark that this Lyapunov function in \eqref{eq:ly-func-unified} may take negative values, depending on the choices of $A$ and $B,$ but it is lower bounded, While we can add a sufficiently large constant to the function to make it, it becomes less intuitive.  Therefore, we keep $V(\cdot)$ in the current form to better reflect the underlying intuition and will comment that the tail bound in \cite{WanMagSri_21} can be directly applied to a lower bounded Lyapunov function.}
The first term in \eqref{eq:ly-func-unified} captures the number of waiting jobs (except the jobs in service), while the second term reflects the number of busy servers (or equivalently, the complement of idle servers). The constants \( A \) and \( B \) are chosen based on intuition from single-server queues under heavy traffic regimes (recall \( \lambda N = N - N^{1 - \alpha} \) and it implies that the number of idle servers is approximated to be \( \Theta(N^{1 - \alpha})\) ideally).  

\begin{itemize}[leftmargin=*]
    \item Sub-Halfin-Whitt regime:  
    In this regime, we expect the number of waiting jobs to be negligible. Hence, we set \( A = 0 \). For \( B \), we take \( B = N_2 + \Theta(N^{1 - \alpha}) \), since we expect that at least the fastest \( N_2 \) servers are busy to serve the incoming load \( \lambda N \), and the additional term accounts for the small number of idle servers.

    \item Super-Halfin-Whitt regime:  
    In this regime, we expect nearly all servers to be busy. Therefore, we set \( B = N \), reflecting full utilization. For \( A \), we use \( A = \Theta(N^{1} \log N) \), consistent with the order of idle servers in the ideal fluid system, with an additional \( \log N \) factor required for technical reasons when invoking tail probability bounds. 
\end{itemize}
Note that the choice of $B$ indicates our belief in the number of busy servers. The values of $B = N_2 + \Theta(N^{1 - \alpha})$ and $B = N$ are both greater and refined than $N_2 - 5\sqrt{N}\log N$ as discussed in Lemma \ref{lem:most n2}.  Next, we specify their values and justify the key state-space collapse results under Sub-Halfin-Whitt and Super-Halfin-Whitt regimes, respectively. 
\begin{figure}[H]
  \centering
\includegraphics[width=6.1in]{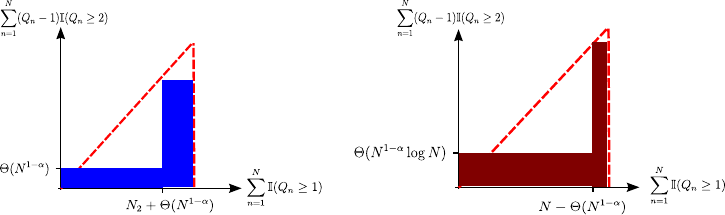}
  \caption{State-space collapse of Sub and Super Halfin-Whitt regimes under JFSQ.}
  \label{fig:SSC}
\end{figure}

\subsubsection*{State-space Collapse in Sub-Halfin-Whitt Regime} 
Under Sub-Halfin-Whitt regime, let $A= 0$ and $B= N_2 + (1-\frac{1}{4\nu})N^{1-\alpha}$ with $\nu = \frac{(b-1)\mu_{N_2}}{\epsilon}$ in Lyapunov function \eqref{eq:ly-func-unified}. 
We have the following lemma. 
\begin{restatable}{lemma}{sscsup}\label{lem:SSC-sup}
 We define a Lyapunov function $V: \mathbb Q \to \mathbb R$ to be
\begin{align}
V(q)=&\min\left\{\sum_{n=1}^{N} (q_n-1)\mathbb{I}(q_n\geq 2), N_2 + (1-\frac{1}{4\nu})N^{1-\alpha}  - \sum_{n=1}^{N} \mathbb{I}(q_n\geq 1) \right\}. \label{ly-func-sub}
\end{align}
Under JFSQ, for a large $N$ {such that $N^{1-2\alpha} \geq \frac{128(r+1)(b-1)^2\mu_{N_2}\log N}{\epsilon^2}$} and $\epsilon \geq \frac{2(r+1)}{\log N}$ in Assumption \ref{assumption: sub}, we have 
\begin{align*}
\mathbb P\left(V(Q)\geq (1-\frac{3}{8\nu}) N^{1-\alpha}\right)\leq  \exp\left(-2r\log N\right).
\end{align*}
\end{restatable}
The lemma establishes a high-probability bound on a carefully designed Lyapunov function that captures the imbalance between waiting jobs and idle servers in the system. Specifically, in the Sub-Halfin-Whitt regime, it shows that the system state concentrates around a region where either the number of waiting jobs or the number of idle servers is small and in the order of $\Theta(N^{1-\alpha})$, but not both can be large simultaneously, as shown in left subfigure in Figure \ref{fig:SSC}. This phenomenon, known as state-space collapse/concentration, reflects the system would tend to operate near full efficiency for most capable servers under the JFSQ policy. Note
the result depends on the level of server heterogeneity. Specifically, when the ratio $\mu_{N_2} / \epsilon$ becomes large, indicating that the fastest servers are more capable than the slowest, this heterogeneity requires a larger system size for the state-space collapse to hold. Moreover, Lemma \ref{lem:SSC-sup} implies two negative effects of the heterogeneity ($\mu_{N_2} / \epsilon$) through the factor of $(1-\frac{3}{8\nu})$: 1) a larger number of waiting jobs; and 2) a smaller number of busy servers.

\subsubsection*{State-space Collapse in Super-Halfin-Whitt Regime}
Under Super-Halfin-Whitt regime, Let $A= k N^\alpha \log N$ with $k = \frac{24\mu_1r(b-1)}{\epsilon}$ and $B= N$ in Lyapunov function \eqref{eq:ly-func-unified}. We have the following lemma. 
\begin{restatable}{lemma}{sscsuper}\label{lem:SSC-super}
We define a Lyapunov function $V: \mathbb Q \to \mathbb R$ to be
\begin{align}
V(q)=&\min\left\{\sum_{n=1}^{N} (q_n-1)\mathbb{I}(q_n\geq 2) - k N^\alpha \log N, ~ N - \sum_{n=1}^{N} \mathbb{I}(q_n\geq 1) \right\}. \label{ly-func}
\end{align}
Let $c_1=\frac{1}{6\mu_{1}}.$ Under JFSQ, for a large $N$ such that $N^{1-\alpha} \geq 12 r \mu_1 \log N,$ we have 
\begin{align*}
  \mathbb P\left(V(Q)\geq 3c_1  N^{1-\alpha}\right)
  \leq  \exp\left(- 2r\log N\right).  
\end{align*} 
\end{restatable}
Similarly, in the Super-Halfin-Whitt regime, the lemma establishes a high-probability bound on a carefully designed Lyapunov function that captures the imbalance between waiting jobs and idle servers in the system. Specifically, it shows that the system state concentrates around a region where either the number of waiting jobs (beyond a threshold of $\Theta(N^\alpha \log N)$) or the number of idle servers is small (at an order of $\Theta(N^{1-\alpha})$), but both cannot be large simultaneously, as shown in the right subfigure in Figure \ref{fig:SSC}. 
The result again depends on the degree of server heterogeneity, particularly through the ratio $\mu_1 / \epsilon$, which quantifies how much faster the best server is compared to the slowest. As heterogeneity increases, a larger system size is required to ensure the concentration result, and the resulting delay bounds worsen through the constant $\mu_1 / \epsilon$. 

It is important to note that Lemmas~\ref{lem:SSC-sup} and~\ref{lem:SSC-super} are quite general and directly focus on the waiting jobs and idle servers across the system. This enables us to analyze the behavior without relying on the assumptions of homogeneity or grouped heterogeneity commonly in the existing literature. 
Based on the results in Lemmas~\ref{lem:SSC-sup} and~\ref{lem:SSC-super}, we have established two key events under the JFSQ load balancing policy: either the number of waiting jobs is small, or the number of idle servers is small. In the first case, our analysis shows that the total number of jobs waiting in queues (except jobs in service) is tightly bounded. In the second case, we obtain sharper bounds on the number of busy servers:  at least $N_2 + \frac{1}{2\nu}N^{1-\alpha}$ and $ N - \frac{1}{2 \mu_1}N^{1-\alpha}$ servers are busy with high probability in sub and super Halfin-Whitt regimes, respectively. 
Note these are more refined than the results in Lemma \ref{lem:most n2} where the most servers remain busy ($N_2 - 5\sqrt{N}\log N$ as discussed in the previous section). Therefore, we can expect resource pooling under JFSQ, so its behavior becomes similar to a single-server queue system, where the key performance metric (e.g., delay) can be deduced there. 
This intuition is formalized and rigorously justified through Stein's method. 

\subsection{Generator Approximation via Stein's Method}
To analyze the steady-state delay performance of the load-balancing system, we adopt the generator approximation method via Stein’s method. This technique allows us to couple the original high-dimensional stochastic system with a simpler, low-dimensional reference system, whose behavior is more tractable. The generator discrepancy between these two systems determines the approximation error of the target metric, which is a test function of queue lengths. Following \cite{LiuYin_20}, we choose the test function to be the clipped $r$-th moment of average queue length at steady-state. Recall that the most of the fastest servers are busy whenever there are waiting jobs in the system  in Lemmas \ref{lem:SSC-sup} and \ref{lem:SSC-super}. This motivates us to consider the fluid model of a simplified single-queue system where the service rate is slightly above the arrival rate. Specifically, the fluid system evolves according to the ODE
\[
\dot{x} = -\frac{1}{N^{\alpha}},
\]
with arrival rate $\lambda$ and effective departure rate $\lambda + \frac{1}{N^\alpha}$. Intuitively, we can associate the term $x$ with the average queue length, and the ODE could be a good approximation for good load balancing algorithms (e.g., JFSQ). 

Our metric of interesting  is the moment bound of total queue length, and we consider the $r$-th moment of the test function $$h(x) = \max\{x - \eta, 0\},$$ for some threshold $\eta > 0.$ The threshold is carefully chosen under different heavy traffic regimes and is specified later. 
Let $g(x)$ denote the solution to the following Stein's equation:
\begin{align}
\frac{dg(x)}{dt} = g'(x) \cdot \left(-\frac{1}{N^{\alpha}}\right) = h^r(x), ~ \forall x, \label{Gen:L}
\end{align}
where $g'(x) = \frac{dg(x)}{dx}.$ 
Note that while the simplified fluid system is one-dimensional, the actual load balancing system is $N$-dimensional. We define the average queue length at steady-state $\bar{Q} := \frac{1}{N} \sum_{n=1}^{N} Q_n$ and its realization  
$\bar{q} := \frac{1}{N} \sum_{n=1}^{N} q_n.$

Let $G$ denote the generator of the original queueing process under load balancing algorithms. $g(\bar Q)$ is bounded according to the \eqref{Gen:L} and the finite buffer assumptioni.e., $\bar Q \leq b$, 
we have the Stein identity:
\begin{equation}
\mathbb{E}[G g(\bar{Q})] = 0. \label{eq:sse}
\end{equation}
Combining the generator identity \eqref{Gen:L} with \eqref{eq:sse}, we derive the coupling equation:
\begin{align}
\mathbb{E}\left[h^r(\bar{Q})\right] 
= \mathbb{E}\left[ g'(\bar{Q}) \cdot \left(-\frac{1}{N^\alpha}\right) - G g(\bar{Q}) \right]. \label{eq:gencou-mainpage}
\end{align}
Under JFSQ, we expand $G g(\bar{Q})$ and have:
\begin{align}
\mathbb{E}\left[h^r(\bar{Q})\right]
=&\mathbb{E}\left[ g'(\bar{Q}) \cdot \left(-\frac{1}{N^\alpha} \right) \right. \nonumber\\
&\left. - \lambda N(1 - A_b(Q)) \left(g\left(\bar{Q} + \frac{1}{N} \right) - g(\bar{Q})\right) \right. \nonumber\\
&\left. - \left(\sum_{n=1}^{N} \mu_n \mathbb{I}(Q_n \geq 1)\right) \left(g\left(\bar{Q} - \frac{1}{N}\right) - g(\bar{Q})\right) \right], \label{gen-diff}
\end{align}
where $A_b(Q)$ denotes the probability that an arriving job joins a queue of length $b$, i.e., blocked. The term of $\lambda N(1-A_b(Q))$ is the rate of the average queue length to increase and the term of $\sum_{n=1}^{N} \mu_n \mathbb{I}(Q_n \geq 1)$ is the rate of a job leaving the system. 
Following the steps in \cite{LiuYin_20} and using Taylor expansion, we derive the iterative relation on $r$-th moment of the average queue length:
\begin{align}
\mathbb{E}\left[h^r\left( \bar Q \right) \right] 
\leq & \mathbb{E} \left[ N^{\alpha} h^r\left(\bar Q \right) \left(\lambda + \frac{1}{N^\alpha} - \frac{1}{N} \sum_{n=1}^{N} \mu_n \mathbb{I}(Q_n \geq 1) \right) \mathbb{I}({\bar{Q} > \eta + \frac{1}{N}}) \right] \label{eq: key}\\
& + \frac{2^{r+2}}{N^{r - \alpha}} 
+ \frac{r}{N^{1 - \alpha}} \mathbb{E} \left[ h^{r-1} \left( \bar{Q} + \frac{1}{N} \right) \right]. \label{eq: key 1}
\end{align}
The full derivation is included in Appendix \ref{app:stein} for completeness. 
If we ignore the term in \eqref{eq: key}, we already have the iteration between the $r$th-order moment and the $(r-1)$th-order moment of queue length. Therefore, the key is to quantify the term in \eqref{eq: key} and $(\lambda + \frac{1}{N^\alpha} - \frac{1}{N} \sum_{n=1}^{N} \mu_n \mathbb{I}(Q_n \geq 1))$ in particular. 
Intuitively, when the average queue length is greater than a threshold ${\bar{Q} > \eta + \frac{1}{N}}$ and almost all servers are busy under JFSQ, the average service rate $\frac{1}{N} \sum_{n=1}^{N} \mu_n \mathbb{I}(Q_n \geq 1)$ should be ``close'' to $\lambda+\frac{1}{N^\alpha} = 1.$
Specifically, if $\frac{1}{N} \sum_{n=1}^{N} \mu_n \mathbb{I}(Q_n \geq 1) \geq 1-\frac{c}{N^\alpha}$ holds for a constant smaller than $c<1$, {we have the key term in \eqref{eq: key} 
$$\mathbb{E} \left[ N^{\alpha} h^r\left(\bar Q \right) \left(\lambda + \frac{1}{N^\alpha} - \frac{1}{N} \sum_{n=1}^{N} \mu_n \mathbb{I}(Q_n \geq 1) \right) \mathbb{I}({\bar{Q} > \eta + \frac{1}{N}}) \right]\leq c \cdot \mathbb{E} \left[ h^r\left(\bar Q \right)\right].$$
This term can be absorbed into $\mathbb{E}\left[h^r\left(\bar Q \right) \right]$ by moving it to the left-hand-side, resulting in  
\begin{align*}
(1-c)\mathbb{E}\left[h^r\left( \bar Q \right) \right] 
\leq \frac{2^{r+2}}{N^{r - \alpha}} 
+ \frac{r}{N^{1 - \alpha}} \mathbb{E} \left[ h^{r-1} \left( \bar{Q} + \frac{1}{N} \right) \right]. 
\end{align*}
We again have a clean iteration for the $r$th-order moment and the $(r-1)$th-order moment of queue length.}
Next, we justify the intuition for sub and super Haffin-Whitt regimes, where we use distinct thresholds $\eta$ for both regimes. 
We first prove the high-order moment bounds in Theorem \ref{thm:main} and then asymptotic zero-waiting results in Theorem \ref{thm:zerodelay}.

\subsection{Proving Theorem \ref{thm:main}}
As discussed above, the key is to establish the lower bound of average service rate, $\frac{1}{N} \sum_{n=1}^{N} \mu_n \mathbb{I}(Q_n \geq 1),$ in \eqref{eq: key}. We first prove the results of Sub-Haffin-Whitt regime in Theorem \ref{thm:main}.   
\subsubsection*{Sub-Haffin-Whitt Regime}
Based on Lemma \ref{lem:SSC-sup}, we define the key event of state-space collapse 
\begin{align}
\mathcal E_{sub}=&\{q~|~ 
\min\{\sum_{n=1}^{N} (q_n-1)\mathbb{I}(q_n\geq 2), N_2+(1-\frac{1}{4\nu})N^{1-\alpha} -\sum_{n=1}^{N} \mathbb{I}(q_n\geq 1 ) \}\leq (1-\frac{3}{8\nu})  N^{{1-\alpha}}\}.
\end{align}
In Sub-Haffin-Whitt regime, we let $\eta = \frac{N_2}{N} + (1-\frac{1}{4\nu})\frac{1}{N^{\alpha}}$ and further define the event $\mathcal E_1$ as follows 
\begin{align}
\mathcal E_{1}=&\{q~|~ \bar q = \frac{1}{N}\sum_{n=1}^{N} q_n> \frac{N_2}{N} + (1-\frac{1}{4\nu})\frac{1}{N^{\alpha}} + \frac{1}{N}\}. 
\end{align}
For the system states $Q \in \mathcal E_1 \cap \mathcal E_{sub},$ we have
\begin{eqnarray*}
\sum_{n=1}^{N} Q_n > 
N_2+ (1-\frac{1}{4\nu})N^{{1-\alpha}} +1, 
\end{eqnarray*} which implies that 
\begin{align*}
\sum_{n=1}^{N} (Q_n-1)\mathbb{I}({Q_n\geq 2}) =& \sum_{n=1}^{N} Q_n - \sum_{n=1}^{N} \mathbb{I}({Q_n\geq 1}) \nonumber \\
>& N_2+(1-\frac{1}{4\nu})N^{1-\alpha}  - \sum_{n=1}^{N} \mathbb{I}({Q_n\geq 1}).
\end{align*} 
Therefore, when the system states $Q \in \mathcal E_1 \cap \mathcal E_{sub}$, we have
\begin{align*}
&\min\{\sum_{n=1}^{N} (Q_n-1)\mathbb{I}(Q_n\geq 2), N_2+(1-\frac{1}{4\nu})N^{1-\alpha} - \sum_{n=1}^{N} \mathbb{I}(Q_n\geq 1 ) \}\\
= &  N_2+(1-\frac{1}{4\nu})N^{1-\alpha}  - \sum_{n=1}^{N} \mathbb{I}(Q_n\geq 1 )\\
\leq &  (1-\frac{3}{8\nu}) N^{1-\alpha}.
\end{align*} This further implies that 
\begin{align}
    \sum_{n=N_2+1}^{N} \mathbb{I}(Q_n\geq 1) = &\sum_{n=1}^{N} \mathbb{I}(Q_n\geq 1)-\sum_{n=1}^{N_2}  \mathbb{I}(Q_n\geq 1) \nonumber\\
    \geq & N_2+ \frac{1}{8\nu} N^{{1-\alpha}} -N_2
     = \frac{1}{8\nu} N^{{1-\alpha}}. \label{eq: busy sub}
\end{align}
Now we can establish the lower bound of the total service rate as follows
\begin{align}
&\sum_{n=1}^{N} \mu_n \mathbb{I}(Q_n\geq 1) \nonumber\\
=&\sum_{n=1}^{N_2} \mu_n \mathbb{I}(Q_n\geq 1)+\sum_{n=N_2+1}^{N} \mu_n \mathbb{I}(Q_n\geq 1) \nonumber\\ 
= &\sum_{n=1}^{N_2} \mu_n - \sum_{n=1}^{N_1} \mu_n\mathbb{I}(Q_n=0) -\sum_{n=N_1+1}^{N_2} \mu_n\mathbb{I}(Q_n=0)+\sum_{n=N_2+1}^{N} \mu_n \mathbb{I}(Q_n\geq 1) \nonumber\\
\geq &  \lambda N -\sum_{n=1}^{N_1} \mu_n\mathbb{I}(Q_n=0) -\sum_{n=N_1+1}^{N_2} \mu_n\mathbb{I}(Q_n=0) +\frac{\epsilon}{8\nu} N^{{1-\alpha}} \nonumber\\
\geq &  \lambda N -\mu_{N_2} -   \mu_1 \left(2\log^2 N +1\right)- 3\mu_{N_1}\sqrt{N}\log N+\frac{\epsilon}{8\nu} N^{{1-\alpha}} \nonumber\\
\geq &  \lambda N +\frac{\epsilon}{16\nu}N^{{1-\alpha}} 
\end{align} where the first inequality holds due to the definition of $\sum_{n=1}^{N_2}\mu_n = \lambda N$ and the inequality \eqref{eq: busy sub}; the second inequality holds according to Lemmas \ref{lem:most n1} and \ref{lem:most n2}; 
the last inequality holds because {Assumption \ref{assumption: sub} implies $N^{1-\alpha} \geq \frac{16\nu}{\epsilon}(\mu_{N_2} + \mu_1 \left(2\log^2 N +1\right) + 3\mu_{N_1}\sqrt{N}\log N)$.} Therefore, when the event $\tilde{\mathcal E}_{sub}$ happens, we have 
\begin{align}
N^{\alpha} \left({\lambda} + \frac{1}{N^{\alpha}}-\frac{1}{N}\left(\sum_{n=1}^{N} \mu_n \mathbb{I}(Q_n\geq 1)\right)\right) 
\leq   N^{\alpha}  \left( \lambda + \frac{1}{N^\alpha} - \lambda - \frac{\epsilon}{16N^{\alpha} } \right) 
=1-\frac{\epsilon}{16\nu}. \label{eq:largeservice sup}
\end{align}
Now we can plug \eqref{eq:largeservice sup} into the \eqref{eq: key} and we have 
\begin{align}
&\frac{\epsilon}{16\nu}\mathbb{E}\left[h^r\left(\frac{1}{N}\sum_{n=1}^{N} Q_n\right)\right]\nonumber\\
\leq 
&2b^r N^{\alpha}\mathbb P\left(\mathcal E_{sub}^c\right)+\frac{2^{r+2}}{ N^{r-\alpha}}+\frac{r}{N^{1-\alpha}} \mathbb{E}\left[{h^{r-1}}\left(\frac{1}{N}\sum_{n=1}^{N} Q_n +{\frac{1}{N}}\right)\right]\nonumber\\
\leq& \frac{2^{r+2}+2}{ N^{r-\alpha}}+\frac{r}{N^{1-\alpha}} \mathbb{E}\left[{h^{r-1}}\left(\frac{1}{N}\sum_{n=1}^{N} Q_n +{\frac{1}{N}}\right)\right],\nonumber
\end{align}
{where the last inequality holds by $\mathbb P(\mathcal E_{sub}^c) \leq \frac{1}{N^{2r}}$ in Lemma \ref{lem:SSC-sup} and $\frac{b^r}{N^{r-\alpha}} \leq 2$ in Assumption \ref{assumption: sub}.}
Recall $k_{sub} = \frac{16\nu}{\epsilon}=\frac{16(b - 1)\mu_{N_2}}{\epsilon^2}$, we have:

\begin{align}
\mathbb{E} \left[ h^r(\bar{Q}) \right]
\leq & \frac{k_{sub} (2^{r+2}+2)}{ N^{r - \alpha} }
+ \frac{k_{sub} r}{ N^{1-\alpha} } \mathbb{E} \left[ h^{r-1} \left( \bar{Q} + \frac{1}{N} \right) \right]. \label{eq:iter sub}
\end{align}
Let us define the following parameters 
\begin{align*}
w_j := \frac{k_{sub} j}{N^{1-\alpha}}, \quad 
z_j := \frac{k_{sub} (2^{j+2}+2)}{N^{j-\alpha}}.
\end{align*}
Starting with \eqref{eq:iter sub}, we iteratively apply the same inequality to the moment of lower order and then have
\begin{align}
\mathbb{E}[h^r(\bar{Q})] 
&\leq z_r + w_r \mathbb{E}[h^{r-1}(\bar{Q} + \tfrac{1}{N})] \nonumber\\
&\leq z_r + w_r z_{r-1} + w_r w_{r-1} \mathbb{E}[h^{r-2}(\bar{Q} + \tfrac{2}{N})] \nonumber\\
&\leq z_r + w_r z_{r-1} + w_r w_{r-1} z_{r-2} + \cdots +  \prod_{j=1}^{r} w_j. \label{eq:rec_expand}
\end{align}
Therefore, we can deduce the following inequality 
\begin{align}
\mathbb{E}[h^r(\bar{Q})] 
\leq \sum_{k=0}^{r-1} \left( \prod_{j=r-k+1}^{r} w_j \right) z_{r-k} + \prod_{j=1}^{r} w_j \label{eq:moment_bound_sum}
\end{align}
Now we bound the two terms above as follows: 
\begin{itemize}[leftmargin=*]
    \item Since $w_j \leq \frac{k_{sub} r}{N^{1-\alpha}}$ for all $j \leq r$, we have
    \[
    \prod_{j=1}^{r} w_j \leq \left( \frac{k_{sub} r}{N^{1-\alpha}} \right)^r. 
    \]
    \item 
    For each term in the summation, we have:
\[
\left( \prod_{j=r-k+1}^{r} w_j \right) z_{r-k} 
\leq \left( \frac{k_{sub} r}{N^{1-\alpha}} \right)^k \cdot \frac{k_{sub} (2^{r-k+2}+2)}{N^{r-k - \alpha}} 
= \frac{k_{sub}^{k+1} r^k(2^{r-k+2}+2)}{N^{r - \alpha + k(1 - \alpha)}} \leq \frac{8k_{sub}^{k+1} r^k 2^{r-k}}{N^{r - \alpha + k(1 - \alpha)}}.
\]
For a large $N$ such that $\frac{k_{sub} r}{4N^{1-\alpha}} < 1$, we have:
\[
\sum_{k=0}^{r-1} \frac{k_{sub}^{k+1} r^k \cdot 2^{r-k+3}}{N^{r-\alpha + k(1-\alpha)}} =  \frac{2^{r+3} k_{sub}}{N^{r-\alpha}} \cdot \sum_{k=0}^{r-1} \left(\frac{k_{sub} r}{N^{1-\alpha}}\right)^{k}
\leq \frac{2^{r+3} k_{sub}}{N^{r-\alpha}} \frac{1}{1 - \frac{k_{sub} r}{2 N^{1-\alpha}}} \leq \frac{ 2^{r+4} k_{sub}}{N^{r-\alpha}}.
\]
\end{itemize}
Finally, we have 
\begin{align}
\mathbb{E}[h^r(\bar{Q})] \leq \left( \frac{k_{sub} r}{N^{1-\alpha}} \right)^r +  \frac{ 2^{r+4} k_{sub}}{N^{r-\alpha}} \leq 17\left( \frac{ 2k_{sub} r}{N^{1-\alpha}} \right)^r.\label{eq:finalres sup} 
\end{align}

\subsection*{Super-Haffin-Whitt Regime}
Based on Lemma \ref{lem:SSC-super}, we define the key event of state-space collapse 
\begin{align}
\mathcal E_{super}=&\{q~|~ 
\min\{\sum_{n=1}^{N} (q_n-1)\mathbb{I}(q_n\geq 2)- k_{super} N^\alpha \log N, N -\sum_{n=1}^{N} \mathbb{I}(q_n\geq 1 ) \}\leq \frac{1}{2\mu_1}  N^{1-\alpha}\}
\end{align}
Note in the super-Haffin-Whitt regime, we let $\eta = 1 + \frac{k_{super}\log N}{N^{\alpha}}$ and define the event $\mathcal E_2$ holds as follows 
\begin{align}
\mathcal E_{2}=&\{q~|~ \bar q = \frac{1}{N}\sum_{n=1}^{N} q_n> \eta + \frac{1}{N}\}
\end{align}
Now define event $\tilde{\cal E}_{super} := \mathcal E_2 \cap \mathcal E_{super}.$ 
For the system states in $\mathcal E_2,$ we have
\begin{align*}
\sum_{n=1}^{N} Q_n =& \sum_{n=1}^{N} (Q_n-1)\mathbb I(Q_n \geq 2) + \sum_{n=1}^{N} \mathbb I(Q_n \geq 1)\\
>& N(\eta + \frac{1}{N}) = N + k_{super} N^{1-\alpha}\log N+1. 
\end{align*} Therefore, when the event $\tilde{\mathcal E}_{super}$ happens, we have  
\begin{align*}
&\min\left\{\sum_{n=1}^{N} (Q_n-1)\mathbb{I}(Q_n\geq 2) - k_{super} N^\alpha \log N, ~ N- \sum_{n=1}^{N} \mathbb{I}(Q_n\geq 1) \right\} \\
=& N - \sum_{n=1}^{N} \mathbb{I}(Q_n\geq 1) = \sum_{n=1}^{N} \mathbb{I}(Q_n = 0) \leq \frac{1}{2\mu_1} N^{1-\alpha}, 
\end{align*}
which implies that 
\begin{align*}
\sum_{n=1}^{N} \mu_n \mathbb{I}(Q_n\geq 1) 
=& \sum_{n=1}^{N} \mu_n - \sum_{n=1}^{N} \mu_n \mathbb{I}(Q_n = 0) \\
\geq& N - \frac{1}{2\mu_1} \mu_1 N^{1-\alpha} 
= N - \frac{1}{2} N^{1-\alpha},
\end{align*}
where the last inequality holds because of the definition of $\sum_{n=1}^{N} \mu_n = N$ and the assumption of $\mu_1 \geq \cdots \geq \mu_n$.  
{Therefore, recalling that $\lambda + \frac{1}{N^{\alpha}} = 1$, on the event $\tilde{\mathcal E}_{\mathrm{super}}$ we have
\begin{align}
\lambda+\frac{1}{N^{\alpha}}-\frac{1}{N}\left(\sum_{n=1}^{N} \mu_n \mathbb{I}(Q_n\geq 1)\right)\leq \lambda+\frac{1}{N^{\alpha}}-\frac{1}{N}(N - \frac{1}{2} N^{1-\alpha}) = \frac{1}{2N^\alpha}.   \label{eq:largeservice super} 
\end{align}}
Now we can plug \eqref{eq:largeservice super} into the \eqref{eq: key} and have  
\begin{align}
&\mathbb{E}\left[h^r\left(\frac{1}{N}\sum_{n=1}^{N} Q_n\right)\right] \nonumber\\
\leq& 2N^\alpha b^r \mathbb P(\tilde{\cal E}^c) + \frac{2^{r+3}}{ N^{r-\alpha}} + \frac{2r}{N^{1-\alpha}} \mathbb{E}\left[{h^{r-1}}\left(\frac{1}{N}\sum_{n=1}^{N} Q_n +{\frac{1}{N}}\right)\right] \nonumber\\
\leq&\frac{2 + 2^{r+3}}{ N^{r-\alpha}} + \frac{2r}{N^{1-\alpha}} \mathbb{E}\left[{h^{r-1}}\left(\frac{1}{N}\sum_{n=1}^{N} Q_n +{\frac{1}{N}}\right)\right] \label{eq：iter super}
\end{align}
{where the last inequality holds by $\mathbb P(\tilde{\cal  E}^c) \leq \frac{1}{N^{2r}}$ in Lemma \ref{lem:SSC-super}
and $\frac{b^r}{N^{r-\alpha}} \leq 2$ in Assumption \ref{assumption: super}}. Now we have established the iterative relationship of high-order moments. The remaining part follows the same proof by expanding \eqref{eq：iter super}, as in Sub-Haffin-Whitt regime in \eqref{eq:largeservice sup}--\eqref{eq:finalres sup}.  

In summary, we have proved the high-order moment bounds of JFSQ under heavy traffic regime in Theorem \ref{thm:main}. We then move to justify the asymptotic zero waiting in Theorem \ref{thm:zerodelay}. 

\subsection{Proving Theorem \ref{thm:zerodelay}}

\subsubsection*{Sub-Haffin-Whitt regime}
In this section, we establish the key metrics under JSFQ, including the blocking probability $p_{\mathcal B},$ the waiting time $\mathbb E[W]$ and the waiting probability $p_{\mathcal W}.$ 

\vspace{3 pt}
{\bf Blocking probability:} We first study the blocking probability $p_{\mathcal B}$ as follows
\begin{align*}
p_{\mathcal B}
\leq& \mathbb P\left(\sum_{n=1}^N Q_n = Nb\right)  
\leq P\left(h^r(\bar Q) \geq (b-\eta)^r\right) \nonumber\\
\leq& \frac{\mathbb E[h^r(\bar Q)]}{(b-\eta)^r} 
\leq 17\left(\frac{2k_{sub} r}{(b-\eta)N^{1 - \alpha}}\right)^r \leq 17\left(\frac{2k_{sub} r}{N^{1 - \alpha}}\right)^r.
\end{align*}
where the first inequality holds by the definition of blocking under JFSQ; the second inequality holds by the definition of $h$ function; the third inequality holds by the Markov inequality; the fourth inequality holds by Theorem \ref{thm:main};
the last inequality holds because $b-\eta \geq 1.$ 

\vspace{3 pt}

{\bf Waiting time:} 
For jobs that are not blocked, the average response time in the system includes waiting time in the buffer $\mathbb E[W]$ and service time $\mathbb E[S]$. 
Let $Q_{\mathcal W} = \sum_{n=1}^N(Q_n-1)\mathbb I(Q_n \geq 2)$ as the total number of waiting jobs in the buffer at steady-state (i.e., the queue length except jobs in service). 
According to Little's law, we have 
\begin{align}
    \mathbb E[W] = \frac{\mathbb E[Q_{\mathcal W}]}{\lambda N(1-p_{\mathcal B})}.  
\end{align}
Let $r=1$ in Theorem \ref{thm:main}, we have the bound on the total queue length  \begin{align*}
\mathbb E\left[\sum_{n=1}^{N} Q_n\right]  \leq N_2 + N^{1-\alpha} +(34k_{sub} + \log N)N^\alpha.
\end{align*}
This implies
\begin{align*}
\mathbb E\left[Q_{\mathcal W}\right] =& \mathbb E\left[\sum_{n=1}^{N} Q_n\right]  - \mathbb E\left[\sum_{n=1}^N\mathbb I(Q_n \geq 1)\right]  \\
\leq& N_2 + N^{1-\alpha} + (34k_{sub} + \log N)N^\alpha - \mathbb E\left[\sum_{n=1}^{N_2}\mathbb I(Q_n \geq 1)\right]\\
= & N^{1-\alpha} + (34k_{sub} + \log N)N^\alpha + \mathbb E\left[\sum_{n=1}^{N_2}\mathbb I(Q_n = 0)\right]\\
\leq& N^{1-\alpha} + (34k_{sub} + \log N)N^\alpha + 6\sqrt{N}\log N
\end{align*}
where the last inequality holds because of Lemma \ref{lem:most n2}, where $\sum_{n=1}^{N_2}\mathbb{I}(Q_n\geq 1) \geq N_2 - 5\sqrt{N} \log N$ holds with a high probability at least $1-\frac{1}{N}$ when $\log N \geq 1/\epsilon.$
Therefore, the expected waiting time is bounded by
\begin{align*}
\mathbb E[W]
\leq& \frac{N^{1-\alpha} + (34 k_{\rm sub} + \log N) N^\alpha + 6 \sqrt{N}\log N}{\lambda N (1-p_{\mathcal B})}. 
\end{align*}
The proof is completed by choosing a large $N^{1-\alpha} \geq 2r k_{sub} \log (17N)$ such that $p_{\mathcal B} \leq 1/N.$

\vspace{3 pt}
{\bf Waiting probability:} 
Recall $\overline{\mathcal W}$ to be the waiting event that a job is not blocked (i.e., a job is routed to a busy server and waits in a buffer) and  $p_{\overline{\mathcal W}}$ to be its steady-state probability. 
Let $T_{Q}$ be the queueing time for the jobs waiting in the buffer {(i.e., the waiting time of jobs not receiving service immediately) and $S$ be the service time of jobs in service}. {According to Little's law (Little’s Law for Time in Queue of Corollary 6.4 in \cite{harchol-balter_2013})}, we have $$\lambda N p_{\overline{\mathcal W}} \mathbb E[T_{Q}] = \mathbb E\left[Q_{\mathcal W}\right] ~~\text{and}~~ \lambda N(1-p_{\mathcal B})\mathbb E[S] = \mathbb E[\sum_{n=1}^N \mathbb I(Q_n \geq 1)],$$ 
Note $\mathbb E[T_{Q}]$ is lower bounded by the service time of $\mathbb E[S]$ according to its definition, i.e., $\mathbb E[T_{Q}] \geq \mathbb E[S]$. 
It implies $$p_{\overline{\mathcal W}} = \frac{\mathbb E\left[Q_{\mathcal W}\right]}{\lambda N \mathbb E[T_{Q}]} \leq \frac{\mathbb E\left[Q_{\mathcal W}\right]}{\lambda N \mathbb E[S]} = (1-p_{\mathcal B}) \frac{\mathbb E\left[Q_{\mathcal W}\right]}{\mathbb E[\sum_{n=1}^N \mathbb I(Q_n \geq 1)]}.$$  
Therefore, we establish
\begin{align*}
    p_{\overline{\mathcal W}} \leq& (1-p_{\mathcal B})\frac{N^{1-\alpha} +(34k_{sub} + \log N)N^\alpha + 6\sqrt{N}\log N}{(1-\frac{1}{N})(N_2 - 5\sqrt{N}\log N)} \\ 
     \leq& \frac{N^{1-\alpha} +(34k_{sub} + \log N)N^\alpha + 6\sqrt{N}\log N}{N_2 - 5\sqrt{N}\log N} \\
    \leq& \frac{N^{1-\alpha} +(34k_{sub} + \log N)N^\alpha + 6\sqrt{N}\log N}{N -\frac{N^{1-\alpha}}{\epsilon} - 5\sqrt{N}\log N}\\
    \leq& \frac{4N^{1-\alpha} +4(34k_{sub} + \log N)N^\alpha + 24\sqrt{N}\log N}{N}
\end{align*}
where the second inequality holds because $p_{\mathcal B} \leq 1/N$; 
{the third inequality follows from the definition of $N_2$, namely $\sum_{n=1}^{N_2}\mu_n = \lambda N$, which implies $\sum_{n=N_2+1}^{N}\mu_n = N^{1-\alpha}$. Since $\mu_n \ge \epsilon$ for any server $n$, we have $(N - N_2)\epsilon \le N^{1-\alpha}$, and hence $N_2 \ge N - \frac{N^{1-\alpha}}{\epsilon}$;}
the last inequality holds for a large $N$ such that $\epsilon \geq 2/N^\alpha$ and $\sqrt{N} \geq 20\log N.$ Finally, a job that is not routed to an idle server is either blocked or wait in a buffer
\begin{align*}
p_{\mathcal W} =&  p_{\mathcal B}+p_{\overline{\mathcal W}} \leq 
\frac{4N^{1-\alpha} + 4(34k_{sub} + \log N)N^\alpha + 24\sqrt{N}\log N + 1}{N}
\end{align*}
This proves the waiting probability of sub-Haffin-Whitt regime in Theorem \ref{thm:zerodelay}.

\subsubsection*{Super-Haffin-Whitt regime}
The blocking probability and waiting time are established by following the same steps above. In particular, we bound the blocking probability 
\begin{align*}
p_{\mathcal B}
\leq& \mathbb P\left(\sum_{n=1}^N Q_n = Nb\right)  
\leq 32\left(\frac{2 r}{N^{(1 - \alpha)}}\right)^r.
\end{align*}

\vspace{3 pt}
{\bf Waiting time:} 
Deriving the waiting time is different and more involved compared to Sup-Haffin-Whitt regime. The underlying reason is that Lemma \ref{lem:most n2} is not refined enough to quantify the number of idle servers when bounding $\mathbb E[Q_{\mathcal W}]$. Instead, we leverage the state-space result in Lemma \ref{lem:SSC-super} to establish a proper upper bound. 
Recall the definition $Q_{\mathcal W} = \sum_{n=1}^N(Q_n-1)\mathbb I(Q_n \geq 2)$ and the event 
\begin{align}
\mathcal E_{super}=&\left\{  
\min\{Q_{\mathcal W}- k_{super} N^\alpha \log N, N - \sum_{n=1}^N\mathbb I(Q_n \geq 1) \}\leq \frac{N^{1-\alpha}}{2\mu_1} \right\}. \nonumber
\end{align}
Define
\begin{align}
\mathcal E_{small}=&\{Q_{\mathcal W}- k_{super} N^\alpha \log N \leq N - \sum_{n=1}^N\mathbb I(Q_n \geq 1)\}. \nonumber  
\end{align}
Let $r=1$ in Theorem \ref{thm:main}, we have the total queue length bounded by
\begin{align}
\mathbb{E}[\max\{Q_{\mathcal W}+\sum_{n=1}^N\mathbb I(Q_n \geq 1)-N - k_{super}N^\alpha\log N,0\}] \leq 64N^\alpha, \nonumber
\end{align}
which directly implies  
\begin{align}
\mathbb{E}[\max\{Q_{\mathcal W}+\sum_{n=1}^N\mathbb I(Q_n \geq 1)-N - k_{super}N^\alpha\log N,0\} \mathbb I(\mathcal E_{super} \cap \tilde{\mathcal E}_{small})] \leq 64N^\alpha. \nonumber
\end{align}
Therefore, we have 
\begin{align}
\mathbb{E}[Q_{\mathcal W} \mathbb I(\mathcal E_{super} \cap \tilde{\mathcal E}_{small})] \leq& \mathbb{E}[(N-\sum_{n=1}^N\mathbb I(Q_n \geq 1))\mathbb I(\mathcal E_{super} \cap \tilde{\mathcal E}_{small})] + k_{super} N^\alpha\log N + 64N^\alpha \nonumber\\
\leq& \frac{N^{1-\alpha}}{2\mu_1}   + k_{super} N^\alpha\log N + 64N^\alpha. \nonumber
\end{align}
where the second inequality holds because of the definition of $\mathcal E_{super}$ and $\mathcal E_{small}$. 
Now we decouple  
\begin{align*}
\mathbb E\left[Q_{\mathcal W}\right] =&~ 
\mathbb E\left[Q_{\mathcal W}\mathbb I(\mathcal E_{super} \cap \mathcal E_{small})\right] + \mathbb E\left[Q_{\mathcal W}\mathbb I(\mathcal E_{super} \cap \tilde{\mathcal E}_{small})\right] + \mathbb E\left[Q_{\mathcal W}\mathbb I(\tilde{\mathcal E}_{super})\right].
\end{align*}
For the first term, according to the definition of $\mathcal E_{super}$ and $\mathcal E_{small},$ we have 
$$\mathbb E\left[Q_{\mathcal W}\mathbb I(\mathcal E_{super} \cap \mathcal E_{small})\right] \leq \frac{N^{1-\alpha}}{2\mu_1}  + k_{super} N^\alpha\log N.$$
For the last term, we use Lemma \ref{lem:SSC-super} and have 
\begin{align*}
\mathbb E\left[Q_{\mathcal W}\mathbb I(\tilde{\mathcal E}_{super})\right] 
\leq&~ Nb \mathbb P\left(\tilde{\mathcal E}_{super}\right) \leq \frac{Nb}{N^{2}} \leq 1. 
\end{align*}
By combining all three terms above, we have
\begin{align}
\mathbb E\left[Q_{\mathcal W}\right] \leq &~ \frac{N^{1-\alpha}}{\mu_1}   + 2 k_{super} N^\alpha\log N + 64N^\alpha + 1. \label{eq:waiting jobs super} 
\end{align}
Therefore, the expected waiting time is bounded by
\begin{align*}
\mathbb E[W] \leq \frac{\frac{N^{1-\alpha}}{\mu_1}   + 2 k_{super} N^\alpha\log N + 64N^\alpha + 1}{\lambda N(1-p_{\mathcal B})}. 
\end{align*}

\vspace{3 pt}
{\bf Waiting probability:}
Recall the upper bound of waiting probability as follows
$$p_{\overline{\mathcal W}} \leq (1-p_{\mathcal B}) \frac{\mathbb E\left[Q_{\mathcal W}\right]}{\mathbb E[\sum_{n=1}^N \mathbb I(Q_n \geq 1)]}.$$ 
Similarly, by \eqref{eq:waiting jobs super} and Lemma \ref{lem:most n2}, we have 
\begin{align*}
p_{\overline{\mathcal W}} \leq& (1-p_{\mathcal B})\frac{\frac{N^{1-\alpha}}{\mu_1} + 2k_{super}N^\alpha\log N + 64N^\alpha + 1}{(1-\frac{1}{N})(N_2 - 5\sqrt{N}\log N)} \\
\leq& \frac{\frac{N^{1-\alpha}}{\mu_1} + 2k_{super}N^\alpha\log N + 64N^\alpha + 1}{(N_2 - 5\sqrt{N}\log N)} \\
\leq& \frac{\frac{4N^{1-\alpha}}{\mu_1} + 8k_{super}N^\alpha\log N + 256N^\alpha + 4}{N} 
\end{align*}
where the second inequality holds because $p_{\mathcal B} \leq 1/N$; the third inequality holds because the lower bound of $N_2$ is obtained when the service rate $\mu_n = \epsilon, \forall n > N_2$, i.e., $(N-N_2)\epsilon = N^{1-\alpha}$; the last inequality holds for a large $N$ such that $\epsilon \geq 2/N^\alpha$ and $\sqrt{N} \geq 20\log N$.

Finally, a job that is not routed to an idle server is either blocked or in a buffer, so 
\begin{align*}
p_{\mathcal W} =&  p_{\mathcal B}+p_{\overline{\mathcal W}} \leq 
\frac{\frac{4N^{1-\alpha}}{\mu_1} + 8k_{super}N^\alpha\log N + 256N^\alpha + 5}{N} 
\end{align*}
This proves the results of Super-Haffin-Whitt regime in Theorem \ref{thm:zerodelay}.

\section{Conclusion}
This paper studied JFSQ load balancing in fully heterogeneous many-server systems in the heavy-traffic regimes. We established that JFSQ achieves asymptotic zero-waiting results with the explicit convergence rates with respect to the system size $N$ and dependence on the server heterogeneity captured by the ratio of service rates between the fastest server and the slowest server. Our analysis combined Lyapunov drift techniques and Stein’s method to show that a high-dimensional heterogeneous system under JFSQ can be coupled by a single-server system. The key in our analysis is to design novel Lyapunov functions to derive iterative state-space peeling into the typical region where the most servers are busy under JFSQ. 

\newpage
\appendix
\section{Tail Bound via Lyapunov Drift Analysis}
We introduce a probability tail bound via Lyapunov drift analysis in Lemma \ref{TailBound} from \cite{WanMagSri_21}, which is an extension of the tail bound in \cite{BerGamTsi_01} and is the key to establishing iterative state space peeling. This lemma allows us to apply the Lyapunov drift analysis to iteratively reduce the state space. 

\begin{lemma}\label{TailBound}
Let $(S(t): t \geq 0)$ be a continuous-time Markov chain over a finite state space $\mathcal S$ and is irreducible, so it has a unique stationary distribution $\pi,$ {i.e., $S(\infty) \sim \pi.$}  Consider a Lyapunov function $V: \mathcal S \to R^{+}$ and define the drift of $V$ at a state $s \in \mathcal S$ as $$\nabla V(s) = \sum_{s' \in \mathcal S: s' \neq s} q_{s,s'} (V(s') - V(s)),$$ where $q_{s,s'}$ is the transition rate from $s$ to $s'.$ Assume
\begin{align*} 
\nu_{\max} :=& \max\limits_{s,s'\in \mathcal S: q_{s,s'} >0} |V(s') - V(s)|< \infty ~~ \text{ and } ~~ \bar q := \max\limits_{s \in \mathcal S} (-q_{s,s}) < \infty 
\end{align*}
and define $$q_{\max}\geq \max\limits_{s \in \mathcal S} \sum_{s' \in \mathcal S: V(s) < V(s')} q_{s,s'}.$$
Assume there exists a set $\mathcal E$ with $B>0$, $\gamma>0$, $\delta \geq 0$ such that the following conditions hold
\begin{itemize}
\item $\nabla V(s) \leq -\gamma$ when $V(s) \geq B$ and $s \in \mathcal E.$
\item $\nabla V(s) \leq \delta$ when $V(s) \geq B$ and $s \notin \mathcal E.$
\end{itemize}
Then $$\mathbb P \left(V(S(\infty))\geq B + 2 \nu_{\max} j \right) \leq \alpha^j + \beta \mathbb P\left(S(\infty) \notin \mathcal E\right), ~ \forall j \in \mathbb{N},$$
holds with $$\alpha = \frac{q_{\max}\nu_{\max}}{q_{\max}\nu_{\max} + \gamma} ~~\text{ and }~~ \beta = \frac{\delta}{\gamma}+1.$$
\end{lemma}

{
\begin{remark}\label{remark:lya}
Lemma \ref{TailBound} from \cite{WanMagSri_21} requires for a nonnegative Lyapunov function but can be more broadly applied to a lower-bounded Lyapunov function. Considering $V(s)$ which is lower bounded and letting $m = \inf_{s\in\mathcal S} V(s)$, we can define a shifted function
$$
W(s) = V(s) - m \geq 0,
$$
and let $D = B - m$. Since $V(s)$ and $W(s)$ differ by a univeral constant for all $s$, the drift remains the same
$\nabla W(s) = \nabla V(s),$
and the one-step differences satisfy $|W(s')-W(s)| = |V(s')-V(s)|$. Therefore, all the conditions of the lemma that hold for $(V,B)$ also hold for $(W,D)$, and the same tail bound follows. Moreover,
$$
\mathbb{P}\big(W(S(\infty)) \geq D + x\big)
= \mathbb{P}\big(V(S(\infty)) \geq B + x\big),
$$
so the probability statement is identical. 
In our analysis, we use  lower-bounded Lyapunov functions defined in the most intuitive way without explicitly making them non-negative. 
\end{remark}
}

\section{Most of the Fastest Servers are Busy}
\mostnone*
\begin{proof}
Note $V(q)\geq 1$ means that at least one of the first $N_1$ servers is idle.  So when an arrival occurs, it goes to one of the first $N_1$ servers under JFSQ, we have
\begin{align*}
\triangledown V(q) = & -\lambda N +\sum_{n=1}^{N_1}\mu_n \mathbb{I}(q_n\geq 1)\\
\leq & -\lambda N + (1-\delta) N\\
= & N^{1-\alpha}-\delta N\\
\leq & -\frac{\delta N}{2}.
\end{align*}
where the second inequality holds according to the definition of $\sum_{n=1}^{N_1} \mu_n = (1-\delta)N;$ the last inequality holds because $2/N^\alpha\leq \delta.$ Now invoking Lemma \ref{TailBound} with $B=1,$ $\gamma= {\delta N}/{2},$ $\nu_{\max}=1,$ $q_{\max}\leq N,$ and $j=4\log^2 N,$ and $\mathcal E$ being the entire state space, we obtain that 
\begin{align*}
  \mathbb P\left(V(q)\geq 1+ 8\log^2 N \right)&\leq \left(\frac{N}{N+\frac{\delta N}{2}}\right)^{4\log^2 N}\\ 
   &\leq \left(1- \frac{\delta}{4}\right)^{ 4\log^2 N}\\
   &\leq \exp\left(- \delta \log^2 N\right).
\end{align*}
where the second inequality holds by the assumption $\delta \leq 0.5$ and the last inequality holds $(1-a)^b\leq e^{-ab}$ for any $a\in (0,1)$ and $b>0.$
\end{proof}

\mostntwo*
\begin{proof}
Note $V(q)\geq 1$ means that at least one of the first $N_2$ servers is idle.  So when an arrival occurs, it goes to one of the first $N_2$ servers under JFSQ. Therefore, we have when $V(q)\geq \sqrt{N}\log N,$
\begin{align*}
\triangledown V(q) = & -\lambda N + \sum_{n=1}^{N_2}\mu_n \mathbb{I}(q_n = 1)\\
\leq & -\lambda N + \sum_{n=1}^{N_2}\mu_n \mathbb{I}(q_n\geq 1)\\
=& -\left(\lambda N -\sum_{n=1}^{N_2}\mu_n +\sum_{n=1}^{N_2}\mu_n \mathbb{I}(q_n=0)\right)\\
= & - \sum_{n=1}^{N_2}\mu_n \mathbb{I}(q_n=0)\\
\leq & - \mu_{N_2} \sum_{n=1}^{N_2} \mathbb{I}(q_n=0)\\
\leq & -\mu_{N_2} \sqrt{N} \log N.
\end{align*} 
where we use the definition of $\lambda N = \sum_{n=1}^{N_2}\mu_n,$ the assumption of monotonicity service rates and $V(q)=\sum_{n=1}^{N_2}\mathbb{I}(q_n=0) \geq \sqrt{N}\log N.$   

Now invoking Lemma \ref{TailBound} with $B=\sqrt{N}{\log N},$ $\gamma=\mu_{N_2} \sqrt{N} \log N,$ $\nu_{\max}=1,$ $q_{\max}\leq N,$ $j=2\sqrt{N} \log N,$ $\delta\leq N,$ and $\mathcal E$ being the entire state space, we obtain that 
\begin{align*}
  \mathbb P\left(V(q)\geq 5\sqrt{N} \log N \right)
   \leq & \left(\frac{N}{N+\mu_{N_2}{\sqrt{N}}\log N}\right)^{2\sqrt{N} \log N}\\ 
   \leq &\left(1- \frac{\mu_{N_2}}{2} \frac{\log N}{\sqrt{N}}\right)^{2\sqrt{N} \log N}\\
   \leq &\exp\left(-\mu_{N_2}\log^2 N\right)\\
   \leq &\exp\left(-\epsilon\log^2 N\right).
\end{align*} 
where the second inequality holds by the assumption $\mu_{N_2} \leq \frac{\sqrt{N}}{\log N}$ and the last inequality holds because $(1-a)^b\leq e^{-ab}$ for any $a\in (0,1)$ and $b>0.$
\end{proof}

\section{State Space Collapse in Sub-Halfin-Whitt Regime}
\sscsup*
\begin{proof}
Let $c_0=1-\frac{3}{4\nu}.$
For the Lyapunov function defined in (\ref{ly-func-sub}), the Lyapunov drift is
\begin{align*}
&\triangledown V(q)=E\left[GV(Q)|Q=q\right]\\
= &\sum_{n=1}^{N}\lambda N \mathbb P_n(q) (V(q + e_n) - V(q)) +  \mu_n \mathbb{I}(q_n\geq 1) (V(q - e_n)-V(q)), \nonumber
\end{align*}
{where $\mathbb P_n(q)$ denotes the probability that an arriving job joins the server $n$ given the queue state $q.$} 

Given $V(q)\geq c_0 {N^{1-\alpha}},$ we consider the following two cases.
\begin{itemize}
\item Case 1: Assume $\sum_{n=1}^{N} (q_n-1)\mathbb{I}(q_n\geq 2)\leq N_2+(1-\frac{1}{4\nu})N^{1-\alpha}- \sum_{n=1}^{N} \mathbb{I}(q_n\geq 1).$
If $q_m = 0$, then 
\begin{align*}
&V(q+e_m) -V(q) \\
=& \min\left\{\sum_{n=1}^{N} (q_n-1)\mathbb{I}(q_n\geq 2),~ N_2 + (1-\frac{1}{4\nu})N^{1-\alpha} - \sum_{n=1}^{N} \mathbb{I}(q_n\geq 1)-1\right\}-V(q)\\
\leq & \sum_{n=1}^{N} (q_n-1)\mathbb{I}(q_n\geq 2)-V(q) =0. 
\end{align*} 
If $q_m = 1,$ then 
\begin{align*}
&V(q+e_m)-V(q) \leq \sum_{n=1}^{N} (q_n-1) \mathbb{I}(q_n\geq 2) +1-V(q) =  1,\\
&V(q-e_m) -V(q) = \sum_{n=1}^{N} (q_n-1)\mathbb{I}(q_n\geq 2) -V(q)=0.
\end{align*} 
If $q_m \geq 2,$ then 
\begin{align*}
&V(q+e_m)-V(q)  = \sum_{n=1}^{ N} (q_n-1) \mathbb{I}(q_n\geq 2) +1 -V(q)=1\\
&V(q-e_m) - V(q)= \sum_{n=1}^{ N} (q_n-1)\mathbb{I}(q_n\geq 2) -1-V(q)=-1.
\end{align*} 
From the assumption in Case 1, it implies 
$$\sum_{n=1}^{N} \mathbb{I}(q_n\geq 1)< N_2+ \frac{1}{2\nu}N^{1-\alpha}<\lambda N + \frac{1}
{2}N^{1-\alpha}  <N.$$
where the second inequality holds because $\nu \geq 1$ and $N_2 \leq \lambda N.$ 

Therefore, at least one server is idle, and the incoming job is routed to a server $m$ with $q_m=0$. Under JFSQ, we have
\begin{align}
\triangledown V(q)\leq& -\sum_{n=1}^{N} \mu_n\mathbb{I}(q_n\geq 2) \nonumber\\
\leq& -\epsilon \sum_{n=1}^{N} \mathbb{I}(q_n\geq 2) \nonumber\\
\leq& -\frac{\epsilon}{b-1}\sum_{n=1}^{N} (q_n-1)\mathbb{I}(q_n\geq 2) \nonumber\\
\leq& - \frac{\epsilon c_0 }{(b-1)} N^{1-\alpha}  \nonumber 
\end{align}
where the first inequality holds according to the definition of Lyapunov drift by plugging in the cases of $q_m=0$, $q_m=1$ and $q_m\geq 2$ above; the second inequality holds by the assumption $\mu_1 \geq \cdots \geq \mu_{N} \geq \epsilon;$ the third inequality holds by the assumption of finite buffer size $b-1;$ the last inequality holds $V(q) = \sum_{n=1}^{N} (q_n-1)\mathbb{I}(q_n\geq 2) \geq c_0 {N^{1-\alpha}}.$ 

\item Case 2:  Assume $\sum_{n=1}^{N} (q_n-1)\mathbb{I}(q_n\geq 2) > N_2 + (1-\frac{1}{4\nu})N^{1-\alpha} - \sum_{n=1}^{N} \mathbb{I}(q_n\geq 1).$ Similarly as above, we have three cases. 

If $q_m = 0,$ then 
\begin{align*}
V(q+e_m) -V(q) = &N_2 + (1-\frac{1}{4\nu})N^{1-\alpha}  - \left(1+\sum_{n=1}^{N} \mathbb{I}(q_n\geq 1)\right)-V(q)=-1.
\end{align*} 
If $q_m = 1,$ then 
\begin{align*}
&V(q+e_m)-V(q) = N_2 + (1-\frac{1}{4\nu})N^{1-\alpha} -\sum_{n=1}^{ N} \mathbb{I}(q_n\geq 1)-V(q)=0\\
&V(q-e_m) -V(q)\leq N_2 + (1-\frac{1}{4\nu})N^{1-\alpha}  - \left(\left(\sum_{n=1}^{N} \mathbb{I}(q_n\geq 1)\right)-1\right)-V(q)=1
\end{align*} 
If $q_m \geq 2,$ then 
\begin{align*}
&V(q+e_m)-V(q)  = N_2 + (1-\frac{1}{4\nu})N^{1-\alpha}  - \sum_{n=1}^{N} \mathbb{I}(q_n\geq 1)-V(q)=0\\
&V(q-e_m) -V(q)\leq N_2 + (1-\frac{1}{4\nu})N^{1-\alpha} - \sum_{n=1}^{N} \mathbb{I}(q_n\geq 1)-V(q)=0.
\end{align*} 

Therefore, we have that under JFSQ, 
\begin{align*}
\triangledown V(q)\leq&-\lambda N + \sum_{m=1}^{N} \mu_m\mathbb{I}(q_m=1)\\
=&-\lambda N + \sum_{m=1}^{ N} \mu_m\mathbb{I}(q_m\geq 1)-\sum_{m=1}^{N} \mu_m\mathbb{I}(q_m\geq 2)\\
=&-\sum_{m=1}^{N_2} \mu_m + \sum_{m=1}^{ N} \mu_m\mathbb{I}(q_m\geq 1)-\sum_{m=1}^{N} \mu_m\mathbb{I}(q_m\geq 2)\\
\leq &\sum_{m=N_2+1}^{ N} \mu_m\mathbb{I}(q_m\geq 1)-\sum_{m=1}^{N_2} \mu_m \mathbb{I}(q_m = 0) - \sum_{m=1}^{N} \mu_m\mathbb{I}(q_m\geq 2).
\end{align*}
To bound the Lyapunov drift, we first study the first term  
\begin{align*}
\sum_{m=N_2+1}^{ N} \mathbb{I}(q_m\geq 1)
= &\sum_{m=1}^{ N} \mathbb{I}(q_m\geq 1)-\sum_{m=1}^{N_2} \mathbb{I}(q_m\geq 1)\\
\leq &N_2 + (1-\frac{1}{4\nu}- c_0) N^{{1-\alpha}}-\sum_{m=1}^{N_2} \mathbb{I}(q_m\geq 1)\\
= & \frac{1}{2\nu}  N^{{1-\alpha}} + \sum_{m=1}^{N_2} \mathbb{I}(q_m = 0)
\end{align*} where the first inequality holds because $N_2+(1-\frac{1}{4\nu})N^{1-\alpha}- \sum_{m=1}^{ N} \mathbb{I}(q_m\geq 1)\geq  c_0  N^{1-\alpha};$ the third equality holds because $c_0=1-\frac{3}{4\nu}.$  
Therefore, we have 
\begin{align*}
\triangledown V(q)
\leq&  \frac{\mu_{N_2}}{2\nu} N^{{1-\alpha}} + \mu_{N_2} \sum_{m=1}^{N_2} \mathbb{I}(q_m = 0) -\sum_{m=1}^{N_2} \mu_m \mathbb{I}(q_m = 0) -\sum_{m=1}^{N} \mu_m\mathbb{I}(q_m\geq 2)\\
\leq& \frac{\mu_{N_2}}{2\nu} N^{1-\alpha}  - \sum_{m=1}^{N} \mu_m\mathbb{I}(q_m\geq 2) \\
\leq& \frac{\mu_{N_2}}{2\nu} N^{1-\alpha}  - \frac{\epsilon c_0 }{(b-1)} N^{1-\alpha} \\
\leq & - \frac{\epsilon}{4(b-1)} N^{1-\alpha}
\end{align*} 
where the last inequality holds because of $\nu = \frac{(b-1)\mu_{N_2}}{\epsilon}$ and $b\geq 5.$  
\end{itemize}

In summary, we have that under JFSQ, when $V(q)\geq c_0 N^{1-\alpha}$ and the event $\mathcal E_0$ occurs,  
\begin{align*}
\triangledown V(q) \leq - \frac{\epsilon}{4(b-1)} N^{1-\alpha}.
\end{align*} When the event $\mathcal E_0$ does not occur,  
\begin{align*}
\triangledown V(q) \leq N.
\end{align*}
{We now apply the tail bound in Lemma~\ref{TailBound}. As noted in Remark~\ref{remark:lya}, although the Lyapunov function $V(\cdot)$ is not nonnegative, it is lower bounded and can be shifted by a large value (e.g., $N-N_2$ in this case) to satisfy the lemma without affecting the drift or the resulting tail bound.}
Invoking Lemma \ref{TailBound} with $B=c_0N^{1-\alpha},$ $\gamma= \frac{ \epsilon}{4(b-1)} N^{1-\alpha},$ $\nu_{\max}=1,$ $q_{\max} = N,$ $\delta = N$ and $j=\frac{1-c_0}{4}N^{1-\alpha},$ we obtain that for sufficiently large $N$
\begin{align*}
  \mathbb P\left(V(Q)\geq \frac{1+c_0}{2} N^{1-\alpha}\right)
   &\leq \left(\frac{N}{N+\frac{\epsilon}{4(b-1)} N^{1-\alpha}}\right)^{\frac{1-c_0}{4} N^{1-\alpha}}+\left(\frac{N}{\frac{\epsilon}{4(b-1)} N^{1-\alpha}}+1\right) \exp\left(-\epsilon \log^2 N\right)\\
   &\leq \left(1-\frac{\epsilon}{8(b-1)N^\alpha}\right)^{\frac{\epsilon}{8(b-1)\mu_{N_2}} N^{1-\alpha}} + \left(\frac{4(b-1)N^\alpha}{\epsilon}+1\right)\exp\left(-\epsilon\log^2 N\right)\\
   &\leq \exp\left(-\frac{\epsilon^2}{64(b-1)^2\mu_{N_2}} N^{1-2\alpha}\right) + \left(\frac{4(b-1)N^\alpha}{\epsilon}+1\right)\exp\left(-\epsilon\log^2 N\right)\\
   &\leq \exp\left(-2(r+1)\log N\right) + \left(\frac{4(b-1)N^\alpha}{\epsilon}+1\right)\exp\left(-2(r+1)\log N\right)\\
   &\leq \exp\left(-2r\log N\right)
\end{align*} 
where the second inequality holds by plugging the value of $c_0$ and $\frac{\epsilon}{2(b-1)} \leq N^{\alpha}$; the third inequality holds because $(1-a)^b\leq e^{-ab}$ for any $a\in (0,1)$ and $b>0;$ the fourth inequality holds because $N^{1-2\alpha} \geq \frac{128(r+1)(b-1)^2\mu_{N_2}\log N}{\epsilon^2}$ and $\epsilon \geq \frac{2(r+1)}{\log N};$ the last inequality holds by $N^{1-\alpha} \geq \frac{4(b-1)}{\epsilon}+2$.
\end{proof}

\section{State Space Collapse in Super-Halfin-Whitt Regime}
\sscsuper*
\begin{proof}
Let $c_1=\frac{1}{6\mu_{1}}.$
For the Lyapunov function defined in (\ref{ly-func}), the Lyapunov drift is
\begin{align*}
&\triangledown V(q)=E\left[GV(Q)|Q=q\right]\\
= &\sum_{n=1}^{N}\lambda N \mathbb P_n(q) (V(q + e_n) - V(q)) +  \mu_n \mathbb{I}(q_n\geq 1) (V(q - e_n)-V(q)), \nonumber
\end{align*}
{where $\mathbb P_n(q)$ is the probability that an arriving job joins the server $n$ given the queue state $q.$} 

Given $V(q)\geq c_1 {N^{1-\alpha}},$ we consider the following two cases.
\begin{itemize}
\item Case 1: Assume $N - \sum_{n=1}^{N} \mathbb{I}(q_n\geq 1) \geq \sum_{n=1}^{N} (q_n-1)\mathbb{I}(q_n\geq 2) - k N^{\alpha}\log N \geq c_1 {N^{1-\alpha}}.$

If $q_m = 0$, then 
\begin{align*}
&V(q+e_m) -V(q) \\
=& \min\left\{\sum_{n=1}^{N} (q_n-1)\mathbb{I}(q_n\geq 2) - k N^\alpha\log N,~ N - \sum_{n=1}^{N} \mathbb{I}(q_n\geq 1)-1\right\}-V(q)\\
\leq & \sum_{n=1}^{N} (q_n-1)\mathbb{I}(q_n\geq 2)- k N^\alpha\log N-V(q) =0. 
\end{align*} 
If $q_m = 1,$ then 
\begin{align*}
&V(q+e_m)-V(q) \leq \sum_{n=1}^{N} (q_n-1) \mathbb{I}(q_n\geq 2)- k N^\alpha\log N+1-V(q) =  1,\\
&V(q-e_m) -V(q) = \sum_{n=1}^{N} (q_n-1)\mathbb{I}(q_n\geq 2)- k N^\alpha\log N-V(q)=0.
\end{align*} 
If $q_m \geq 2,$ then 
\begin{align*}
&V(q+e_m)-V(q)  = \sum_{n=1}^{ N} (q_n-1) \mathbb{I}(q_n\geq 2)- k N^\alpha\log N+1 -V(q)=1\\
&V(q-e_m) - V(q)= \sum_{n=1}^{ N} (q_n-1)\mathbb{I}(q_n\geq 2)- k N^\alpha\log N-1-V(q)=-1.
\end{align*} 
From the assumption in Case 1, it implies 
$$N > N - c_1 N^{1-\alpha} >  \sum_{n=1}^{N} \mathbb{I}(q_n\geq 1).$$ Therefore, at least one server is idle, and the incoming job is routed to a server $m$ with $q_m=0$. Under JFSQ, we have
\begin{align}
\triangledown V(q)\leq& -\sum_{n=1}^{N} \mu_n\mathbb{I}(q_n\geq 2) \nonumber\\
\leq& -\epsilon \sum_{n=1}^{N} \mathbb{I}(q_n\geq 2) \nonumber\\
\leq& -\frac{\epsilon}{b-1}\sum_{n=1}^{N} (q_n-1)\mathbb{I}(q_n\geq 2) \nonumber\\
\leq& - \frac{\epsilon }{(b-1)} (c_1 N^{1-\alpha} + k N^{\alpha}\log N) \label{eq: low}
\end{align}
where the first inequality holds according to the definition of Lyapunov drift by plugging in the cases of $q_m=0$, $q_m=1$ and $q_m\geq 2$ above; the second inequality holds by the assumption $\mu_1 \geq \cdots \geq \mu_{N} \geq \epsilon;$ the third inequality holds by the assumption of finite buffer size $b-1;$ the last inequality holds by the condition in the assumption in Case 1 above.      

\item Case 2: Assume $\sum_{n=1}^{N} (q_n-1)\mathbb{I}(q_n\geq 2) - k N^{\alpha}\log N >  N - \sum_{n=1}^{N} \mathbb{I}(q_n\geq 1) \geq c_1 {N^{1-\alpha}}.$
 
If $q_m = 0,$ then 
\begin{align*}
V(q+e_m) -V(q) = &N - \left(1+\sum_{n=1}^{N} \mathbb{I}(q_n\geq 1)\right)-V(q)=-1.
\end{align*} 
If $q_m = 1,$ then 
\begin{align*}
&V(q+e_m)-V(q) = N -\sum_{n=1}^{ N} \mathbb{I}(q_n\geq 1)-V(q)=0,\\
&V(q-e_m) -V(q)\leq N   - \left(\left(\sum_{n=1}^{N} \mathbb{I}(q_n\geq 1)\right)-1\right)-V(q)=1.
\end{align*} 
If $q_m \geq 2,$ then 
\begin{align*}
&V(q+e_m)-V(q)  = N   - \sum_{n=1}^{N} \mathbb{I}(q_n\geq 1)-V(q)=0\\
&V(q-e_m) -V(q)\leq N  - \sum_{n=1}^{N} \mathbb{I}(q_n\geq 1)-V(q)=0.
\end{align*} 

Therefore, we have that under JFSQ, 
\begin{align*}
\triangledown V(s)\leq&-\lambda N + \sum_{n=1}^{N} \mu_n\mathbb{I}(q_n=1)\\
=&-\lambda N + \sum_{n=1}^{N} \mu_n\mathbb{I}(q_n\geq 1)-\sum_{n=1}^{N} \mu_n\mathbb{I}(q_n\geq 2)\\
\leq & N(1-\lambda)-\sum_{n=1}^{N} \mu_n\mathbb{I}(q_n\geq 2)\\
\leq& N(1-\lambda) - \frac{\epsilon }{(b-1)} (c_1 N^{1-\alpha} + k N^{\alpha}\log N) \\
\leq&  - \frac{\epsilon }{2(b-1)} (c_1 N^{1-\alpha} + k N^{\alpha}\log N)
\end{align*}
where the first inequality holds according to the definition of Lyapunov drift by plugging in the cases of $q_m=0$, $q_m=1$ and $q_m\geq 2$ above; the second inequality holds by the fact $\sum_{m=1}^{N} \mu_m = N$; the third inequality holds following the same steps as in \eqref{eq: low}; the last inequality holds by the fact $N(1-\lambda) = N^{1-\alpha}$ and the value of $k = \frac{24\mu_1r(b-1)}{\epsilon}.$ 
\end{itemize}

Now invoking Lemma \ref{TailBound} with $B=c_1N^{1-\alpha},$ $\gamma= \frac{\epsilon k}{2(b-1)} N^{\alpha}\log N,$ $\nu_{\max}=1,$ $q_{\max}\leq N,$ $\delta=N$ and $j= c_1 N^{1-\alpha},$ we obtain that for sufficiently large $N$
\begin{align*}
  \mathbb P\left(V(Q)\geq 3c_1  N^{1-\alpha}\right)
   &\leq \left(\frac{N}{N+\frac{\epsilon k}{2(b-1)} N^{\alpha}\log N}\right)^{c_1 N^{1-\alpha}} \\
   &= \left(1 - \frac{\epsilon k}{4(b-1)} N^{\alpha - 1} \log N \right)^{-c_1 N^{1-\alpha}} \\
   &\leq \exp\left(- \frac{\epsilon k c_1}{2(b-1)} \log N\right)
\end{align*}
Recall the value of $k = \frac{24\mu_1r(b-1)}{\epsilon}.$ The second and last inequality holds because $N^{1-\alpha} \geq \frac{\epsilon k}{2(b-1)}  \log N = 12 r \mu_1  \log N.$ We finally have
\begin{align*}
  \mathbb P\left(V(Q)\geq 3c_1  N^{1-\alpha}\right)
   \leq  \exp\left(-r \log N\right).
\end{align*}
\end{proof}

\section{Stein's Method and Iterative Moment Bounds}\label{app:stein}
We provide the detailed steps in generator approximation via Stein's method and derive the iterative moment bounds. To understand the steady-state performance of a load-balancing algorithm, we will establish moment bounds on the following function \cite{LiuYin_20}:
$$h(Q) = \max\left\{\frac{1}{N}\sum_{n=1}^{N} Q_n- \eta,0\right\},$$
where $\eta$ is positive and is different for sub and super Haffin-Whitt regimes. We consider a simple and unified fluid system with arrival rate $\lambda$ and departure rate $\lambda+\frac{1}{N^{\alpha}},$ for the heavy-traffic regime 
$$\dot{x}=-\frac{1}{N^{\alpha}},$$ and function $g(x)$ which is the solution of the following Stein's equation \cite{Yin_16}:
\begin{align}
\frac{dg(x)}{dt} = g'(x) \left(-\frac{1}{N^{\alpha}}\right)  = \left(\max\left\{x-\eta,0\right\}\right)^r, \ \forall x, \label{Gen:L app}
\end{align} where {$g'(x) = \frac{d g(x)}{d x}$} and $r$ is a positive integer. The left-hand side of \eqref{Gen:L app} is to apply the generator of the simple fluid system to function $g(x),$ i.e. $$\frac{dg(x)}{dt}=g'(x)\dot{x}=g'(x) \left(-\frac{1}{N^{\alpha}}\right).$$
It is easy to verify that the solution to \eqref{Gen:L app} is
\begin{equation}g(x)=-\frac{N^{\alpha}}{(r+1)}\left(x-\eta\right)^{r+1} \mathbb{I}_{x\geq\eta},\label{eq:L}\end{equation} and
\begin{equation}g'(x)=-N^{\alpha}\left(x-\eta\right)^{r} \mathbb{I}_{x\geq\eta}.\label{eq:Lder}
\end{equation}
We note that the simple fluid system is a one-dimensional system and the stochastic system is $N$-dimensional. 
Define the average queue length $\bar q= \frac{\sum_{n=1}^Nq_n}{N}$ and $\bar Q= \frac{\sum_{n=1}^NQ_n}{N}$ for simple notation.  
In order to couple these two systems, we define
\begin{equation}
f(q)=g\left(\bar q\right),\label{eq:steinsolution}
\end{equation} and use $f(q)$ defined above in Stein's method. Since $f(q)$ is bounded with a finite buffer, we have 
\begin{equation}
\mathbb{E}[Gf(Q)]= \mathbb E\left[Gg\left(\bar Q\right)\right]=0.\label{eq:sse app}
        \end{equation}
Based on \eqref{Gen:L app} and \eqref{eq:sse app}, we obtain
\begin{align}
\mathbb{E}\left[{h^r}\left(\bar Q\right)\right] =& \mathbb{E}\left[ g'\left(\bar Q\right)  \left(-\frac{1}{N^{\alpha}}\right) - Gg\left(\bar Q\right)\right].\label{eq:gencou app}
\end{align}
Note that according to the definition of $f(q)$ in \eqref{eq:steinsolution} and $e_n$ for $n\leq N$, we have $$f(q+e_n)=g\left(\bar q+\frac{1}{N}\right)$$ and
$$f(q-e_n)=g\left(\bar q-\frac{1}{N}\mathbb{I}(q_n\geq 1)\right).$$ Therefore,
\begin{align*}
    Gg\left(\bar q\right)=&
    N\lambda(1-A_b(q))\left(g\left(\bar q+\frac{1}{N}\right)-g\left(\bar q\right)\right)\\
    &+\left(\sum_{n=1}^{N} \mu_n \mathbb{I}(q_n\geq 1)\right) \left(g\left(\bar q-\frac{1}{N}\right)-g\left(\bar q\right)\right),
\end{align*} where $A_b(q)$ is the probability that an incoming job is routed to a server with $b$ jobs. 
Substituting the equation above to \eqref{eq:gencou app}, we have \begin{align}
\mathbb{E}\left[{h^r}\left(\bar Q\right)\right]
=&\mathbb{E}\left[g'\left(\bar Q\right) \left(-\frac{1}{N^{\alpha}}\right)-N\lambda(1-A_b(Q))\left(g\left(\bar Q+\frac{1}{N}\right)-g\left(\bar Q\right)\right)\right.\nonumber\\
    &\left.-\left(\sum_{n=1}^{N} \mu_n \mathbb{I}(q_n\geq 1)\right) \left(g\left(\bar Q-\frac{1}{N}\right)-g\left(\bar Q\right)\right)\right]. \label{gen-diff-expand}
\end{align}
From the closed-forms of $g$ and $g'$ in \eqref{eq:L} and \eqref{eq:Lder}, note that for any $x <  \eta,$
\begin{equation*}
g(x) = g'\left(x\right)=0.
\end{equation*} Also note that when $x>\eta+\frac{1}{N},$
$$g'(x)=- N^{\alpha} \left(x-\eta\right)^{r},$$
\begin{equation}g''(x)=- r N^{\alpha}\left(x-\eta\right)^{r-1} .\end{equation}
By using mean-value theorem in the region $[\eta-\frac{1}{N}, \eta+\frac{1}{N}]$ and Taylor theorem in the region $(\eta+\frac{1}{N}, \infty)$, we have
\begin{align}
g(x+\frac{1}{N})-g\left(x\right) =& \left(g(x+\frac{1}{N})-g\left(x\right)\right) \left(\mathbb I_{\eta-\frac{1}{N} \leq x \leq \eta+\frac{1}{N}} + \mathbb I_{ x > \eta+\frac{1}{N}} \right) \nonumber\\
=&  \frac{g'(\xi)}{N}\mathbb I_{\eta-\frac{1}{N} \leq x \leq \eta+\frac{1}{N}} + \left(\frac{g'(x)}{N} +  \frac{g''(\zeta)}{2N^2}\right) \mathbb I_{ x > \eta+\frac{1}{N}} \label{g+1/N}\\
g(x-\frac{1}{N})-g\left(x\right) =& \left(g(x-\frac{1}{N})-g\left(x\right)\right) \left(\mathbb I_{\eta-\frac{1}{N} \leq x \leq \eta+\frac{1}{N}} + \mathbb I_{ x > \eta+\frac{1}{N}} \right) \nonumber\\
=&  -\frac{g'(\tilde{\xi})}{N}\mathbb I_{\eta-\frac{1}{N} \leq x \leq \eta+\frac{1}{N}} + \left(-\frac{g'(x)}{N} +  \frac{g''(\tilde{\zeta})}{2N^2}\right) \mathbb I_{ x > \eta+\frac{1}{N}} \label{g-1/N}
\end{align}
where $\xi, \zeta \in (x,x+\frac{1}{N})$ and $\tilde{\xi}, \tilde{\zeta} \in (x-\frac{1}{N},x).$
Substitute \eqref{g+1/N} and \eqref{g-1/N} into the generator difference in \eqref{gen-diff-expand}, we have
\begin{align}
&\mathbb{E}\left[h^r\left(\bar Q\right)\right]\nonumber\\
= &\mathbb{E}\left[g'\left(\bar Q\right)\left(\lambda A_b(Q)- \lambda-\frac{1}{N^{\alpha}}+\frac{1}{N}\left(\sum_{n=1}^{N} \mu_n \mathbb{I}(Q_n\geq 1)\right) \right)\mathbb{I}_{\bar Q> \eta +\frac{1}{N}}\right] \label{G-expansion-SSC}\\
&+\mathbb{E}\left[\left(g'\left(\bar Q\right)\left(-\frac{1}{N^{\alpha}}\right)-\lambda(1-A_b(Q))g'(\xi) + \frac{1}{N}\left(\sum_{n=1}^{N} \mu_n \mathbb{I}(Q_n\geq 1)\right)g'(\tilde{\xi})\right)\mathbb{I}_{\eta-\frac{1}{N} \leq \bar Q \leq \eta+\frac{1}{N}}\right] \label{G-expansion-Gradient-1}\\
&-\mathbb{E}\left[\frac{1}{2N}\left(\lambda(1-A_b(Q))g''(\zeta) + \frac{1}{N}\left(\sum_{n=1}^{N} \mu_n \mathbb{I}(Q_n\geq 1)\right)g''(\tilde{\zeta})\right)\mathbb{I}_{\bar Q > \eta+\frac{1}{N}}\right]. \label{G-expansion-Gradient-2}
\end{align}
Next, we study $g'$ and $g''$ to bound the terms \eqref{G-expansion-Gradient-1} and \eqref{G-expansion-Gradient-2}, and SSC to bound the term \eqref{G-expansion-SSC}.

\subsection{Gradient Bounds}
We summarize bounds on $g'$ and $g''$ in the following two lemmas.

\begin{lemma}\label{lemma:g'}
For any  $x\in\left[\eta -\frac{2}{N}, \eta  +\frac{2}{N}\right],$ we have $$|g'(x)|\leq \frac{2^r}{ N^{r-\alpha}}.$$
\end{lemma}
\begin{proof}
For any $x\in\left[\eta  -\frac{2}{N}, \eta +\frac{2}{N}\right],$ from the closed-form expression of  $g'$ in \eqref{eq:Lder} we have 
\begin{eqnarray*}
|g'(x)|=& \left|- N^{\alpha}\left(x-\eta\right)^{r} \mathbb{I}_{x\geq \eta} \right|\\
\leq& N^{\alpha}\left(\frac{2}{N}\right)^r=\frac{2^r}{N^{r - \alpha}}.    
\end{eqnarray*}
\end{proof}

\begin{lemma}\label{lemma:g''}
For $x>\eta,$ we have 
\begin{align*}
|g''(x)|\leq r N^{\alpha} h^{r-1}(x).
\end{align*}
\end{lemma}
\begin{proof}
For $x>\eta,$ we have
$$g'(x)=- N^{\alpha} \left(x-\eta\right)^{r} ,$$ which implies
\begin{eqnarray*}
g''(x)= -r N^{\alpha} \left(x-\eta\right)^{r-1}.
\end{eqnarray*}
and
\begin{eqnarray*}
|g''(x)|\leq r N^{\alpha} \left(\max\left\{x-\eta,0\right\}\right)^{r-1}.
\end{eqnarray*}
\end{proof}

Based on Lemma \ref{lemma:g'}, we bound the term \eqref{G-expansion-Gradient-1}
\begin{align}
\eqref{G-expansion-Gradient-1}\leq& \left({\lambda+1}+\frac{1}{{N}^{\alpha}}\right)\frac{2^r}{N^{r-\alpha}} \leq \frac{2^{r+1}}{ N^{r-\alpha}}. \label{Bound-1}
\end{align}

Based on Lemma \ref{lemma:g''} and fact that $h(x)$ is a nondecreasing function, we bound the term \eqref{G-expansion-Gradient-2}
\begin{align}
 \eqref{G-expansion-Gradient-2}\leq&  \mathbb{E}\left[\frac{1}{2N}\left(\lambda |g''(\zeta)| + \frac{1}{N}\left(\sum_{n=1}^{N} \mu_n \mathbb{I}(Q_n\geq 1)\right)|g''(\tilde{\zeta})|\right)\mathbb{I}_{\bar Q > \eta+\frac{1}{N}}\right]\\
 \leq &\frac{1}{2N} rN^{\alpha} \mathbb{E}\left[{h^{r-1}}\left(\bar Q +{\frac{1}{N}}\right)\left(\lambda + \frac{1}{N}\left(\sum_{n=1}^{N} \mu_n \mathbb{I}(Q_n\geq 1)\right)\right)\right]\\
  \leq &\frac{\lambda+1}{2N} rN^{\alpha} \mathbb{E}\left[{h^{r-1}}\left(\bar Q +{\frac{1}{N}}\right)\right] \\
 \leq & \frac{r}{N^{1-\alpha}} \mathbb{E}\left[{h^{r-1}}\left(\bar Q +{\frac{1}{N}}\right)\right]. \label{Bound-2}
\end{align}
Therefore, we have 
\begin{align}
&\mathbb{E}\left[g'\left(\bar Q\right)\left(\lambda A_b(Q)- \lambda-\frac{1}{N^{\alpha}}+\frac{1}{N}\left(\sum_{n=1}^{N} \mu_n \mathbb{I}(Q_n\geq 1)\right) \right)\mathbb{I}_{\bar Q> \eta +\frac{1}{N}}\right]\\
=&\mathbb{E}\left[-N^{\alpha} \left(\bar Q-\eta\right)^{r} \left(\lambda A_b(Q)- \lambda-\frac{1}{N^{\alpha}}+\frac{1}{N}\left(\sum_{n=1}^{N} \mu_n \mathbb{I}(Q_n\geq 1)\right) \right)\mathbb{I}_{\bar Q> \eta +\frac{1}{N}}\right] \\
=&\mathbb{E}\left[N^{\alpha} \left(\bar Q-\eta\right)^{r} \left(\lambda+\frac{1}{N^{\alpha}}-\frac{1}{N}\left(\sum_{n=1}^{N} \mu_n \mathbb{I}(Q_n\geq 1)\right) -\lambda A_b(Q) \right)\mathbb{I}_{\bar Q> \eta +\frac{1}{N}}\right] \\
\leq &\mathbb{E}\left[N^{\alpha} h^r(\bar Q) \left(\lambda+\frac{1}{N^{\alpha}}-\frac{1}{N}\left(\sum_{n=1}^{N} \mu_n \mathbb{I}(Q_n\geq 1)\right)  \right)\mathbb{I}_{\bar Q> \eta +\frac{1}{N}}\right].
\end{align}
{where the last inequality holds by dropping the negative term of $-\lambda A_b(Q)$ and recalling the definition of $h(\bar Q)=\max\{\bar Q - \eta, 0\}$.}  
Finally, we have \begin{align}
\mathbb{E}\left[h^r\left(\bar Q\right)\right]
\leq &\mathbb{E}\left[h^r(\bar Q) N^{\alpha} \left(\lambda+\frac{1}{N^{\alpha}}-\frac{1}{N}\left(\sum_{n=1}^{N} \mu_n \mathbb{I}(Q_n\geq 1)\right)  \right)\mathbb{I}_{\bar Q> \eta +\frac{1}{N}}\right] \nonumber\\
&+ \frac{2^{r+1}}{ N^{r-\alpha}} + \frac{r}{N^{1-\alpha}} \mathbb{E}\left[{h^{r-1}}\left(\bar Q +{\frac{1}{N}}\right)\right]. 
\end{align}

\newpage
\bibliographystyle{plain}
\bibliography{inlab-refs}

\end{document}